\RequirePackage{snapshot}
\documentclass[11pt, a4paper, reqno]{amsart}
 
\usepackage{amsmath}
    \allowdisplaybreaks[4]
\usepackage{amssymb}
\usepackage{amsthm}
    \theoremstyle{plain}
    \newtheorem{thm}{Theorem}[section]
    \newtheorem{prop}[thm]{Proposition}
    
    \newtheorem{cor}[thm]{Corollary}
    \newtheorem*{thm*}{Theorem}

    \theoremstyle{definition}
    \newtheorem{definition}[thm]{Definition}
    
    \newtheorem{ex}[thm]{Example}

\usepackage{leftidx} 
 
\usepackage[largesc]{newpxtext}
\usepackage[vvarbb]{newpxmath}  
\usepackage[cal=esstix, scr=boondox, bb=pazo]{mathalfa}

\usepackage{anyfontsize} 

\DeclareFontFamily{U}{matha}{\hyphenchar\font45}
\DeclareFontShape{U}{matha}{m}{n}{
      <5> <6> <7> <8> <9> <10> gen * matha
      <10.95> matha10 <12> <14.4> <17.28> <20.74> <24.88> matha12
      }{}
\DeclareSymbolFont{matha}{U}{matha}{m}{n}
\DeclareFontSubstitution{U}{matha}{m}{n}
\DeclareMathSymbol{\dashV}         {3}{matha}{"2D}

\usepackage[shortlabels]{enumitem} 

\usepackage{mathtools} 

\usepackage{graphicx}
    \newcommand{\figdir}{pstex}
    
\usepackage{pgf,tikz}
    \usetikzlibrary{matrix}
    \usetikzlibrary{arrows, shapes, positioning}
    \tikzset{
        >=stealth',
        punkt/.style={
            rectangle,
            rounded corners,
            draw=black, very thick,
            text width=6.5em,
            minimum height=2em,
            text centered},
        pil/.style={
            ->,
            thick,
            shorten <=2pt,
            shorten >=2pt,}
    }

\usepackage{quiver}
    \tikzcdset{arrow style=tikz, diagrams={>=stealth'}}

\usepackage[american, british]{babel}
\usepackage[english=british]{csquotes}

\usepackage[bibstyle=ieee, citestyle=numeric,
    backref=true, 
    backend=biber, 
    hyperref=true, giveninits=true, sorting=nyt, language=british]{biblatex}
    \DefineBibliographyStrings{english}{%
        backrefpage = {cited on p.~},
        backrefpages = {cited on pp.~},
        url = {},
    } 
    \DefineBibliographyExtras{british}{%
        \stdpunctuation
    }

\usepackage{xpatch}
    \xpatchfieldformat{eprint:arxiv}{\space}{}{}{}
    \xpatchfieldformat{doi}{\space}{}{}{}
    \xpatchfieldformat{url}{\addcolon\space}{}{}{}

\makeatletter
\let\old@font@info\@font@info
\def\@font@info#1{%
\expandafter\ifx\csname\detokenize{#1}\endcsname\relax
  \old@font@info{#1}%
\fi
\expandafter\xdef\csname\detokenize{#1}\endcsname{}%
}
\makeatother

\usepackage{hyperref}	 	
\usepackage[]{microtype}

\newcommand{\cA}{\mathcal{A}}
\newcommand{\cB}{{\mathcal B}}
\newcommand{\cC}{\mathcal{C}}
\newcommand{\cD}{{\mathcal D}}

\newcommand{\DD}{\cD}
\newcommand{\cE}{{\mathcal E}}
\newcommand{\cX}{\mathcal{X}}
\newcommand{\cY}{\mathcal{Y}}

\DeclareMathOperator{\id}{id}
\DeclareMathOperator{\Id}{Id}
\DeclareMathOperator{\Nat}{Nat}
\DeclareMathOperator{\Hom}{Hom}

\DeclareMathOperator{\ev}{ev}
\DeclareMathOperator{\coev}{coev}
\newcommand{\op}{\mathrm{op}}

\newcommand{\AdjL}{\mathcal{A}\mathrm{dj}_{\mathrm{l}}}
\newcommand{\AdjR}{\mathcal{A}\mathrm{dj}_{\mathrm{r}}}

\newcommand{\TwadjLL}{\mathcal{A}\mathrm{dj}^{\mathrm{2L}}_{\mathrm{l}}}
\newcommand{\TwadjLR}{\mathcal{A}\mathrm{dj}^{\mathrm{2L}}_{\mathrm{r}}}

\newcommand{\C}{\mathbb C}
\newcommand{\bicat}{\mathbb}
\newcommand{\CAT}{\bicat{C}\mathrm{AT}}
\newcommand{\PROF}{\bicat{P}\mathrm{ROF}}

\numberwithin{equation}{section}

\renewcommand{\phi}{\varphi}
\newcommand{\define}{\textbf}

\newcommand{\nattrans}{\Rightarrow}
\newcommand{\extranat}{\mathrel{\mkern5mu  \vcenter{\hbox{$\shortmid$}}%
                    \mkern-10mu{\Rightarrow}}}
\newcommand{\profto}{\relbar\joinrel\mapstochar\joinrel\rightarrow} 

\newcommand{\termcat}{\{\star\}}
\newcommand{\companion}[1]{{#1}_\mathrm{c}}
\newcommand{\conjoint}[1]{{#1}^\mathrm{c}}

\DeclareMathOperator{\Vect}{\mathrm{Vect}}

\DeclareMathOperator{\Set}{\mathrm{Set}}

\newcommand{\one}{\mathbb{1}}
\newcommand{\monstr}[1]{\omega}
\newcommand{\mate}[1]{\overline{#1}}

\newcommand{\internalhoml}[2]{{[#1, #2]}_\mathrm{L}}
\newcommand{\internalhomr}[2]{{[#1, #2]}_\mathrm{R}}
\newcommand{\internalhom}[2]{{[#1, #2]}}
\newcommand{\leftlolli}{\multimap}
\newcommand{\rightlolli}{\multimapinv}
\newcommand{\lollihomr}[2]{#2 \multimapinv #1}
\newcommand{\lollihoml}[2]{#1 \multimap #2}
\newcommand{\leftsidedadj}{\dashV_\mathrm{L}}
\newcommand{\rightsidedadj}{\dashV_\mathrm{R}}
\newcommand{\jL}{j_\mathrm{l}}
\newcommand{\jR}{j_\mathrm{r}}
\newcommand{\secondc}{a}
\newcommand{\hatbeta}{\hat{\beta}}

\newcommand{\isoarrow}{\xrightarrow{\raisebox{-0.7ex}[0ex][0ex]{$\sim$}}}

\newcommand*\cocolon{%
        \nobreak
        \mskip6mu plus1mu
        \mathpunct{}%
        \nonscript
        \mkern-\thinmuskip
        {:}%
        \mskip2mu
        \relax
}

\begin{document}

\title[Extranatural transformations and conjugation]{Extranatural transformations,\\adjunctions of two variables\\and conjugation}
\date{}
\author{Simon Willerton}
\address{University of Sheffield}
\email{s.willerton@sheffield.ac.uk}

\begin{abstract}
    Adjunctions of two variables generalize the relationship between tensor product and the internal hom functor in a closed monoidal category. For a pair of ordinary adjunctions $(F\dashv U, F'\dashv U')$ conjugation relates natural transformations of the form $F\Rightarrow F'$ with natural transformations of the form $U' \Rightarrow U$. We look at conjugation for general two variable adjunctions. It is useful in the context of Grothendieck's six operations as we will show that this is an appropriate way to view the constructions of Fausk, Hu and May where they discuss things like the projection formula and internal adjunctions.

    Extensive use is made of surface diagram notation as this is a helpful way to keep track of the three dimensions of composition. This also places the work in the context of formal category theory as, for instance, closed monoidal categories are defined without reference to objects or morphisms inside them. In an appendix it is explained how an appropriate setting for this perspective is the monoidal double category of functors and profunctors, using the fact that the bicategory of profunctors has duals.
\end{abstract}

\maketitle

\setcounter{tocdepth}{1}
\tableofcontents
\thispagestyle{empty}
\enlargethispage*{2em}

\section{Introduction}
\label{sec:intro}

We start with the initial motivation for this paper, which was to give a good categorical framework for a phenomenon which crops up in various places, as noted by Fausk, Hu and May~\cite{FauskHuMay:Isomorphisms}, namely that of certain pairs of isomorphisms that occur in the context of closed monoidal categories, in particular in the context of Grothendieck's six operations.  The introduction then goes on to give an overview of the ideas of the paper.  The perspective is that of formal category theory but it is written with the intention of being accessible to the readers of Fausk, Hu and May's paper; the more formal categorical context, involving the monoidal equipment of profunctors, has been placed in Appendix~\ref{app:double-category}.

A further motivation for this work was to develop the diagrams as a book-keeping device for structural maps that arise in the area of Hopf monads and this should appear in work with Christos Aravanis~\cite{AravanisWillerton:Hopf-monads} building on work from his thesis~\cite{Aravanis:thesis}.

\subsection{The conjugate pairs of Fausk, Hu and May}
We will start by looking at two examples of the phenomenom we will be interested in, which were highlighted by Fausk, Hu and May~\cite{FauskHuMay:Isomorphisms}.

\begin{ex}[Strong monoidal functor and internal hom]
    \label{ex:strong-mon-functor-internal-hom}
    Suppose that \(f^\ast\colon \cY \to \cX\) is a monoidal functor between monoidal categories.  By definition this is equipped with a structual natural transformation $\omega$ with components of the form
    \[
        f^\ast(y) \otimes f^\ast(y') 
        \xrightarrow{\omega_{y,y'}} f^\ast(y \otimes y').
    \]
    If this natural transformation is an isomorphism then \(f^\ast\) is 
    \define{strong monoidal}.

    Suppose further that \(f^\ast\) has a right adjoint \(f_\ast\), so \(f^\ast \dashv f_\ast\), and also both \(\cX\) and \(\cY\) are \emph{closed} monoidal, with internal homs denoted with \(\internalhom{\cdot}{\cdot}\), then you can use this data, together with the above natural transformation to construct a canonical natural transformation $\overline\omega$ with components of the form
    \[
        \internalhom{y}{f_\ast(x)} 
        \xrightarrow{\overline{\omega}_{x, y}} 
        f_\ast\bigl( \internalhom{f^\ast(y)}{x}\bigr).
    \]
    If this natural transformation $\overline{\omega}$ is an isomorphism then we say that \(f^\ast\) and \(f_\ast\) form an \define{internal adjunction}.
\end{ex}

\begin{ex}[Projection formula and strong closed monoidal functors]

    Suppose that, as above, \(f^\ast\colon \cY \to \cX\) is a monoidal functor between monoidal categories, however, now it has a \emph{left} adjoint \(f_!\), so \(f_! \dashv f^\ast\).  Using the structural natural transformation $\omega$ from above and the adjunctions we can construct a natural transformation $\pi$ with components of the form
    \[
        f_!\bigl(x \otimes f^\ast (y)\bigr)
        \xrightarrow{\pi_{x, y}}
        f_!(x) \otimes y.
    \]

    If this natural transformation is an isomorphism then it is said that the \define{projection formula} holds.  (In logic, the term `Frobenius formula' or `Frobenius reciprocity' is sometimes used.)  

    In the case where the categories \(\cX\) and \(\cY\) are actually closed monoidal, so there are internal hom functors for both \(\cX\) and \(\cY\).  Then from the monoidal structure natural transformation and the various adjunctions one can also construct a canonical natural transformation $\overline{\pi}$ with components of the form
    \[
        f^\ast\bigl(\internalhom{y}{y'} \bigr)
        \xrightarrow{\overline\pi_{y, y'}}
        \internalhom{f^\ast(y)}{f^\ast(y')}.
    \]
    If this natural transformation is an isomorphism then we can say that \(f^\ast\) is \define{strong closed monoidal}.

    Again, it is fact that the first of this pair of natural transformations is an isomorphism if and only if the second is an isomorphism.
\end{ex}

\enlargethispage*{2em}
In both of the above two situations, Fausk, Hu and May describe the two natural transformations as being ``conjugate'', though they do not define what they mean by this.  Their terminology seems to be alluding to the notion of conjugate natural transformations which occurs when you have a pair of adjunctions.  They say the following~\cite[p.~111]{FauskHuMay:Isomorphisms}.

\begin{quote}
    ``
    Systematic recognition of such conjugate pairs of isomorphisms can substitute for quite a bit of excess verbiage in the older literature. We call this a `comparison of adjoints' and henceforward leave the details of such arguments to the reader.''
\end{quote}
One of the goals of this paper is to provide an appropriate categorical formulation of this, by showing that what is happening in these cases is conjugation for adjunctions of two variables and to provide an appropriate diagrammatic language for this.  Further categorical context for the diagrammatic language is in Appendix~\ref{app:double-category}.
 
\subsection{Classical adjunctions, conjugation and string diagrams}  
(More details on this are in Section~\ref{sec:adjunctions-and-conjugation}.)
If you have two adjunctions, \(F\colon \cC \rightleftarrows \cD \cocolon U\) and \(F'\colon \cC \rightleftarrows \cD \cocolon U'\), then each natural transformation $\theta \colon F \nattrans F'$
corresponds to -- or is \define{conjugate} to -- a natural transformation $\phi \colon U' \nattrans U$.
 
Conjugation of a natural transformation can be defined using the units -- $\eta$ and $\eta'$ -- and counits -- $\epsilon$ and $\epsilon'$ -- of the two adjunctions.  So, given $\theta$, the components of \(\phi\) can be defined as the following composite:
\[
    U'(d) 
    \xrightarrow{\eta_{U'(d)}} 
    U \circ F \circ U'(d)
    \xrightarrow{U \circ \theta_{U'(d)}}
    U \circ F' \circ U'(d)
    \xrightarrow{U \circ \epsilon'(d)}
    U(d).
\]
In a component-free fashion $\phi$ is written, of course, as follows:
\begin{align*}
    \phi &= (U\circ \epsilon') \odot (U\circ \theta \circ U') \odot (\eta \circ U')
\end{align*}

Here we are working in the \(2\)-cateogry of categories, using two orthogonal types of composition, so it is sensible to switch from a linear notation to a \(2\)-dimensional notation.  We will use string diagrams here -- rather than, say, pasting diagrams -- as the relationship between algbebra and the topology of its representation allows the use of visual reasoning.  The convention here will be that diagrams are read from right to left and bottom to top; this is indicated 
below.  
The conjugate, \(\phi\), of \(\theta\) is represented as follows.
\[
    \phi = \raisebox{-.45\height}{\input{\figdir/conjugation_intro.pstex_t}}
    \quad\qquad 
    \raisebox{-.45\height}{\begin{picture}(0,0)%
\includegraphics{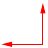}%
\end{picture}%
%
%
\setlength{\unitlength}{4144sp}%
\begingroup\makeatletter\ifx\SetFigFont\undefined%
\gdef\SetFigFont#1#2{%
  \fontsize{#1}{#2pt}%
  \selectfont}%
\fi\endgroup%
\begin{picture}(354,354)(1474,-43)
\end{picture}%
}
\]

Various nice properties of conjugation are easily seen diagrammatically, these are given in Section~\ref{sec:adjunctions-and-conjugation}; a key property is that a natural transformation $\theta\colon F\nattrans F'$ is invertible if and only if its conjugate $\phi \colon U' \nattrans U$ is invertible.

Here we are moving away from talking about components of natural transformations and towards arguing with string diagrams -- below we will use surface diagrams.  This is taking a more \emph{formal category theory} approach where we are thinking about categories as being \(0\)-cells in the \(2\)-category of categories, and using functors (\(1\)-cells) and natural transformations (\(2\)-cells) in definitions and arguments rather than the internal structure of categories such as objects and morphisms.  This perspective can be clarifying, although not always the right one if you want to do actual calculations; importantly it is useful when you want to do something like category theory in other \(2\)-categories, for instance in the \(2\)-cateogry of categories enriched over some base.  This perspective is analogous to the one that is adopted when you first do category theory and, for instance, phrase things about sets, such as the definition of a monoid, in terms of sets and functions -- i.e., objects and morphisms in the category of sets -- rather than in terms of elements of sets.



\subsection{Surface diagrams}
\label{subsec:surface-diagrams}
\enlargethispage*{3em}

We will be working with the monoidal \(2\)-cateogry of categories, where the monoidal product is $\times$, the cartesian product.  This product gives a third orthogonal type of composition; it is possible and appropriate to represent this with a third orthogonal direction in our diagrams, this direction will be into the plane of the paper or screen.  So a functor $K \colon \cC \times \cD \to \cB$ will be represented as pictured below, with the product direction being read from front to back, illustrated in the middle picture.  On the left it is drawn in a box to give a sense of the directions, but in general, the box won't be drawn, as on the right.
\[
    \raisebox{-.45\height}{\input{\figdir/sample_monoidal_one_cell.pstex_t}}
    \qquad
    \raisebox{-.45\height}{\begin{picture}(0,0)%
\includegraphics{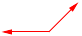}%
\end{picture}%
%
%
\setlength{\unitlength}{4144sp}%
\begingroup\makeatletter\ifx\SetFigFont\undefined%
\gdef\SetFigFont#1#2{%
  \fontsize{#1}{#2pt}%
  \selectfont}%
\fi\endgroup%
\begin{picture}(609,264)(1069,-538)
\end{picture}%
}
    \qquad
    \raisebox{-.45\height}{\input{\figdir/sample_monoidal_one_cell_without_box.pstex_t}}
\]

Natural transformations are then given as going vertically between such pictures.  For example, if we have monoidal categories $\cC$ and $\cD$ with tensor products $\otimes_{\cC} \colon \cC \times \cC \to \cC$ and $\otimes_{\cD} \colon \cD \times \cD \to \cD$ and we have a functor $F\colon \cC \to \cD$, then a natural transformation $\omega$ which has components $\omega_{c, c'}
\colon F(c) \otimes_\cD F(c') \to F(c \otimes_\cC c')$ is of the form $\omega\colon {\otimes_\cD} \circ (F\times F) \nattrans F\circ {\otimes_\cC}$ and so would be pictured as the following.  Again, the boxed form is given here to help orient you and help you get used to the directions.
\[
    \raisebox{-.45\height}{\input{\figdir/sample_monoidal_two_cell_boxed.pstex_t}}
    \qquad
    \raisebox{-.45\height}{\begin{picture}(0,0)%
\includegraphics{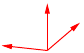}%
\end{picture}%
%
%
\setlength{\unitlength}{4144sp}%
\begingroup\makeatletter\ifx\SetFigFont\undefined%
\gdef\SetFigFont#1#2{%
  \fontsize{#1}{#2pt}%
  \selectfont}%
\fi\endgroup%
\begin{picture}(622,384)(1080,-523)
\end{picture}%
}
\]
In practice, to reduce clutter, labels will not always be put on diagrams when they are clear from the context.

One extra feature of the $2$-category of categories that we wish to capture in the diagrammatic notation is that of opposite category.  Given a functor \(H\colon\cC\to \cD\), we
obtain a functor \(H^\op\colon\cC^\op \to \cD^\op\) between the opposite categories, going in the same direction.  However, given
a natural transformation such as \(\alpha \colon H \circ G \nattrans K\)
we obtain a natural transformation \(\alpha^\op\colon K^\op\nattrans 
H^\op \circ G^\op\) going in the opposite direction.  
In our surface diagrams, we can denote parts of the surface which correspond to an opposite \(\cC^\op\) by shading that part of a surface and labelling it with \(\cC\).  From the above we get that ``shaded'' natural transformations are ``turned
upside-down'':
\[
    \raisebox{-.45\height}{\input{\figdir/sample_nat_trans.pstex_t}}
    \quad\text{gives rise to}\quad 
    \raisebox{-.45\height}{\input{\figdir/sample_op_nat_trans.pstex_t}}
    \quad\text{denoted}\quad
    \raisebox{-.45\height}{\input{\figdir/sample_op_shade_nat_trans.pstex_t}}.
\]

Suppose now we are in the context of closed monoidal categories.  Consider the natural transformation \(\overline\omega\) of Example~\ref{ex:strong-mon-functor-internal-hom}.  It will be convenient to switch notation for internal hom from \(\internalhom{c}{c'}\) to the ``infix lollipop notation'' \(c \leftlolli c'\).  We think of the internal hom as a functor \(\leftlolli\colon \cC^\op \times \cC \to \cC\).  We write the components of \(\overline\omega\) in the form \(\overline\omega_{x,y} \colon (\lollihoml{y}{f_\ast(x)}) \to f_\ast\bigl( \lollihoml{f^\ast(y)}{x}\bigr)\) and write \(\overline\omega\) as a natural transformation \({\leftlolli} \circ (\id \times f_\ast) \nattrans f_\ast \circ {\leftlolli} \circ (f^\ast \times \id)\), so we represent it in surface diagrams as follows.
\[
    \raisebox{-.45\height}{\input{\figdir/omega_bar.pstex_t}}
\]

\subsection{Closed monoidal categories and extranatural transformations}
In order to get the abstraction we want for Fausk, Hu and May's conjugate pairs, we will need to define the notion of closed monoidal category in a more formal fashion than is typical, i.e., avoiding referring to objects in categories, so avoiding referring to the components of evaluation and coevaluation maps.  This is what we do in Section~\ref{sec:closed-monoidal}.
The standard way to express the relationship between tensor product \({\otimes} \colon \cC \times \cC \to \cC\) and internal hom \({\leftlolli} \colon \cC^\op \times \cC \to \cC\) is to say that there are families of evaluation and coevaluation maps, \(\{\ev_{c, c'} \colon c \otimes (c \leftlolli c') \to c\}_{c,c'}\) and  \(\{\coev_{c, c'} \colon c  \to c' \leftlolli (c' \otimes c)\}_{c,c'}\), which satisfy triangle identities.  A sticking point is that these families do not form natural transformations: in each case, the domain and codomain functors do not have the same domain and codomain.  However, there is an alternative way of describing evaluation and coevaluation, which is as \define{extranatural transformations}; we will use '\(\extranat\)' to denote such things.  We can then write evaluation and coevaluation as \(\ev \colon {\otimes} \circ (\id \times {\leftlolli}) \extranat \id\) and \(\coev \colon \id \extranat {\leftlolli} \circ (\id \times {\otimes})\) and represent them with surface diagrams, so for instance we can represent the evaluation extranatural transformation as follows.
\[
    \raisebox{-.45\height}{\input{\figdir/ev_boxed.pstex_t}}
    \qquad
    \raisebox{-.45\height}{\begin{picture}(0,0)%
\includegraphics{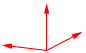}%
\end{picture}%
%
%
\setlength{\unitlength}{4144sp}%
\begingroup\makeatletter\ifx\SetFigFont\undefined%
\gdef\SetFigFont#1#2{%
  \fontsize{#1}{#2pt}%
  \selectfont}%
\fi\endgroup%
\begin{picture}(640,384)(1085,-523)
\end{picture}%
}
\]
This is like how we drew a natural transformation, remembering that we are going from bottom to top, except that we are allowing the right-most face to not just consist of vertical lines but also of cups and caps.

Note that here we are beyond the world of monoidal \(2\)-categories, this does not represent a \(2\)-cell in the monoidal \(2\)-category of categories as it is not a natural transformation.  So in order to define things that follow we need extra structure for the monoidal \(2\)-category of categories.  This is not an essential point for the story here, but is of more general interest.  One way to think of these extranatural transformations, following Street, is as coming from the monoidal bicategory of profunctors where every object -- every category -- has a dual -- its opposite.  The surface diagrams can be interpreted as representing \(2\)-cells in the monoidal equipment (or double category) of functors and profunctors.  This is explained in  Appendix~\ref{app:double-category}.

We can then use this structure to define, in Section~\ref{sec:adjunctions-of-two-variables}, (left-sided) adjunctions of two-variables which are more-or-less the same as parametrized adjoints.  The data consists of two functors \(T \colon \cA \times \cB \to \cC\) and \(H \colon \cA^{\op} \times \cC \to \cB\) -- think tensor and hom -- together with two extranatural transformations, the counit and unit,
\[
    \epsilon \colon T \circ ({\id_\cA} \times H) \extranat \id_\cC
    \quad\text{and}\quad
    \eta \colon \id_\cB \extranat H \circ ({\id_{\cA^\op}} \times T).
\]
These satisfy triangle identities such as the following.
\begin{align*}
    \raisebox{-.45\height}{\input{\figdir/triangle_1_lhs.pstex_t}}
    ~&=~
    \raisebox{-.45\height}{\input{\figdir/triangle_1_rhs.pstex_t}}
\end{align*}
We write \(T\leftsidedadj H\).

If \(\cA\) is the terminal category then this recovers the usual notion of adjunction.  Also new adjunctions of two variables can be created by composing adjunctions of two variables with specific functors and ordinary adjunctions.  A certain case is spelled out in Section~\ref{subsec:new-adjunctions-from-old}.

\subsection{Conjugation for adjunctions of two variables}

Conjugation can be defined for adjunctions of two variables in the sense that if \(T\leftsidedadj H\) and $T' \leftsidedadj H'$ are two adjunctions of two variables and \(\theta \colon T \nattrans T'\) is a natural transformation then there is a conjugate natural transformation \(\phi \colon H' \nattrans H\) which is defined in the following way, analogous to that in the ordinary adjunction case.
\[
    \phi \coloneq \raisebox{-.45\height}{\input{\figdir/two_var_conjugate_forward.pstex_t}}
\]
Using this framework we can then, in Section~\ref{sec:examples}, easily identifty all of the `conjugate pairs' of Fausk, Hu and May as being conjugate in this sense.  They also have a notion of `conjugate triple' and the approach here, where we don't assume symmetry, makes it clear that this comes from two conjugate pairs which can be combined when symmetry is present. 

\subsection{A double-categorical context}

In Appendix~\ref{app:double-category}, following Street~\cite{StreetFunctorialCalculus} and others, we delve further into some categorical context for this work. 
Doing ordinary category theory with categories, functors and natural transformations can be thought of as working in the \(2\)-category of categories, \(\CAT\), and various aspects can be translated to an arbitrary bicategory, \(\bicat{C}\).  For instance, adjunctions can be defined in an arbitrary bicategory.  Translating other aspects of category theory requires more structure on the bicategory \(\bicat{C}\); for instance, when thinking about monoidal categories, you need \(\bicat{C}\) to be a \emph{monoidal} bicategory, and then a monoidal category translates as a pseudo-monoid in \(\bicat{C}\).  You can then ask what structure on \(\bicat{C}\) is needed then in order to translate a closed monoidal category or, indeed, an adjunction of two variables.

The \emph{microcosm principle} of Baez and Dolan~\cite{BaezDolan:HDA-III} suggests that you might think about defining a ``closed pseudo-monoid'' in a monoidal bicategory \(\bicat{C}\) when \(\bicat{C}\) is in some sense a ``closed monoidal bicategory''.  However, it is not clear that the \(2\)-category of categories \(\CAT\) is closed monoidal in an appropriate sense, and that is not the approach taken here.  Rather, \(\CAT\) has an ``extension'', the monoidal bicategory of profunctors \(\PROF\) which is closed monoidal in the sense that it has duals, and we view extranatural transformations in \(\CAT\) as shadows cast by this structure in \(\PROF\).

So one should be able to define ``extraordinary \(2\)-cells'' in a monoidal bicategory \(\bicat{C}\) if it has an ``extension'' \(\bicat{P}\) which is a monoidal bicategory with duals.  Here an ``extension'' means that \(\bicat{C}\) and \(\bicat{P}\) form an equipment (in particular a double category) in the same way that \(\CAT\) and \(\PROF\) do.  In the appendix we interpret our surface diagrams as representing \(2\)-cells in the monoidal double category of functors and profunctors. In the setting of \(\bicat{C}\) forming an equipment with such a \(\bicat{P}\) it would be possible to define adjunctions of two variables and closed  pseudo-monoids in \(\bicat{C}\) as we have done for \(\CAT\) here.  I am not aware that the details of a ``monoidal equipment with duals in one direction'' have yet been fully worked out.

\subsection{Acknowledgements}
Thanks to Christos Aravanis for many conversations.  
Commutative diagrams were prepared with quiver~\cite{Arkor:quiver} and typeset with tikzcd.  Surface diagrams were drawn in Xfig~\cite{xfig}.

\section{Classical adjunctions and conjugation}
\label{sec:adjunctions-and-conjugation}
This section is a warm-up.  We recall the operation of conjugation in the context of pairs of adjunctions and in particular how it is represented string diagrammatically.  The point is that once this is understood in the right fashion, the generalization in Section~\ref{sec:adjunctions-of-two-variables} will seem straightforward.  A classical treatment of conjugation is given in MacLane's book~\cite[Chapter~IV.7]{MacLaneCategoriesWork}.

First we recall the convention for string diagrams.  If  \(F\colon  \cC \rightleftarrows \cD \cocolon U\) forms an adjunction then we denote the unit \(\eta \colon \id \nattrans U \circ F\) and the counit \(\epsilon \colon F \circ U \nattrans \id\) as follows, where diagrams are read right to left and bottom to top.
\[
    \eta\equiv{\raisebox{-.45\height}{\input{\figdir/adj_unit.pstex_t}}};
    \quad
    \epsilon\equiv\raisebox{-.45\height}{\input{\figdir/adj_counit.pstex_t}}.
\]
The zig-zag identities, or triangle identities, are represented as follows:
\[
    \raisebox{-.45\height}{\input{\figdir/zig_zag_1a.pstex_t}} = \raisebox{-.45\height}{\begin{picture}(0,0)%
\includegraphics{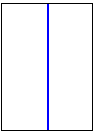}%
\end{picture}%
\setlength{\unitlength}{4144sp}%
\begingroup\makeatletter\ifx\SetFigFont\undefined%
\gdef\SetFigFont#1#2{%
  \fontsize{#1}{#2pt}%
  \selectfont}%
\fi\endgroup%
\begin{picture}(717,1013)(5353,-2871)
\put(5782,-2798){\makebox(0,0)[lb]{\smash{{\SetFigFont{8}{9.6}{\color[rgb]{0,0,0}$U$}%
}}}}
\end{picture}%
}\, ;
    \qquad
    \raisebox{-.45\height}{\begin{picture}(0,0)%
\includegraphics{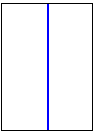}%
\end{picture}%
\setlength{\unitlength}{4144sp}%
\begingroup\makeatletter\ifx\SetFigFont\undefined%
\gdef\SetFigFont#1#2{%
  \fontsize{#1}{#2pt}%
  \selectfont}%
\fi\endgroup%
\begin{picture}(718,1014)(5353,-2872)
\put(5782,-2799){\makebox(0,0)[lb]{\smash{{\SetFigFont{8}{9.6}{\color[rgb]{0,0,0}$F$}%
}}}}
\end{picture}%
} = \raisebox{-.45\height}{\input{\figdir/zig_zag_2b.pstex_t}}\, .
\]

There are two categories which we can take as `the category of adjunctions' and conjugation provides an isomorphism between them.  Conjugation is part of the more general mates correspondence which provides an isomorphism between two double categories of adjunctions -- see the work of Kelly and Street~\cite{KellyStreetReview2Cats} -- but we will just use the conjugation aspect (this can be thought of as restricting to the horizontal categories of the double categories).  This conjugation will be lifted to the context of two-variable adjunctions later.

Given two adjunctions, we can define the conjugation operation between sets of natural transformations.
\begin{definition}
    \label{defn:ordinary_conjugation}
    For adjunctions \(F \dashv U\) and \(F' \dashv U'\) we can define a pair of maps
    \[
        \begin{tikzcd}[ampersand replacement=\&, column sep=tiny]
            {\Nat(F,F')} \&  \& {\Nat(U', U).}
            \arrow["\jL", shift left, curve={height=-1pt}, from=1-1, to=1-3]
            \arrow["{\jR}", shift left, curve={height=-1pt}, from=1-3, to=1-1]
        \end{tikzcd}
    \]
    We define \(\jL\) in the following way.  In terms of components for \(\theta \in \Nat(F, F')\) and \(d \in \cD\), we take \(\jL(\theta)_d \colon U'(d) \to U(d)\) to be the composite
    \[
         U'(d) \xrightarrow{\eta_{U'(d)}} U\circ F \circ U'(d)
         \xrightarrow{U\circ \theta \circ U'} U\circ F' \circ U'(d)
         \xrightarrow{U\circ \epsilon_{d}} U(d).
    \]
    In terms of natural transformations \(\jL(\phi)\in \Nat(U', U)\) is the composite
    \[
        U' \xRightarrow{\eta \circ U'}
        U\circ F \circ U'
        \xRightarrow{U\circ \theta \circ U'}
        U \circ F' \circ U'
        \xRightarrow{U\circ \epsilon'}
        U.
    \]
    So in diagrammatic notation, we can define \(\jL\) as follows.
    \[
        \raisebox{-.45\height}{\input{\figdir/classic_mates_1.pstex_t}}
        \mapsto~
        \raisebox{-.45\height}{\input{\figdir/classic_mates_2.pstex_t}}~.
    \]
    We define the other map \(\jR\) analogously so that for \(\phi\in \Nat(U', U)\) the conjugate \(\jR(\phi)\) is the composite
    \[
        F \xRightarrow{F \circ \eta'}
        F\circ U' \circ F'
        \xRightarrow{F\circ \phi \circ F'}
        F \circ U \circ F'
        \xRightarrow{ \epsilon \circ F'}
        F'.
    \]
    So in diagrammatic notation, we can define \(\jR\) as follows.
    \[
        \raisebox{-.45\height}{\input{\figdir/classic_mates_4.pstex_t}}
        \mapsfrom~
        \raisebox{-.45\height}{\input{\figdir/classic_mates_3.pstex_t}}~.
    \]
\end{definition}
We will give proof that the the maps defined above are inverse to each other, as the generalized two-variable version in Section~\ref{subsec:general-conjucation-two-variables} will be the same.
\begin{prop}
    \label{prop:ordinary-conjugation-inverse}
    The two maps defined above are mutually inverse.
\end{prop}
\begin{proof} We just show one way round.
    \[
        \jR\circ \jL(\theta)
        ~=~
        \raisebox{-.45\height}{\begin{picture}(0,0)%
\includegraphics{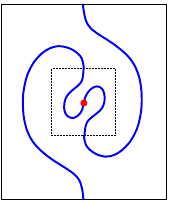}%
\end{picture}%
%
%
\setlength{\unitlength}{4144sp}%
\begingroup\makeatletter\ifx\SetFigFont\undefined%
\gdef\SetFigFont#1#2{%
  \fontsize{#1}{#2pt}%
  \selectfont}%
\fi\endgroup%
\begin{picture}(1283,1533)(4986,-3414)
\put(5603,-2704){\makebox(0,0)[rb]{\smash{{\SetFigFont{8}{9.6}{\color[rgb]{0,0,0}$\theta$}%
}}}}
\end{picture}%
}
        ~=~
        \raisebox{-.45\height}{\begin{picture}(0,0)%
\includegraphics{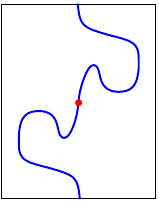}%
\end{picture}%
%
%
\setlength{\unitlength}{4144sp}%
\begingroup\makeatletter\ifx\SetFigFont\undefined%
\gdef\SetFigFont#1#2{%
  \fontsize{#1}{#2pt}%
  \selectfont}%
\fi\endgroup%
\begin{picture}(1197,1535)(5026,-3416)
\end{picture}%
}
        ~=~
        \raisebox{-.45\height}{\begin{picture}(0,0)%
\includegraphics{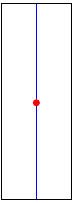}%
\end{picture}%
\setlength{\unitlength}{4144sp}%
\begingroup\makeatletter\ifx\SetFigFont\undefined%
\gdef\SetFigFont#1#2{%
  \fontsize{#1}{#2pt}%
  \selectfont}%
\fi\endgroup%
\begin{picture}(554,1535)(5348,-3416)
\end{picture}%
}
        ~=~
        \theta
    \]
\end{proof}

\begin{definition}
    \begin{enumerate}
        \item The \define{left category of adjunctions} \(\AdjL\) is the category where
        \begin{itemize}
            \item the objects are adjunctions \(F\colon  \cC \rightleftarrows \cD \cocolon U\), which we will also write as \(F \dashv U\) if we don't need to emphasise the domains and codomains,
            \item the morphisms are the natural transformations between the \emph{left} adjoint functors,
            \[
                \AdjL(F\dashv U, F'\dashv U') \coloneq \Nat(F, F') .
            \]
        \end{itemize}
        \item The \define{right category of adjunctions} \(\AdjR\) is the category where
        \begin{itemize}
            \item the objects are adjunctions \(F\colon  \cC \rightleftarrows \cD \cocolon U\), as in \(\AdjL\) above,
            \item the morphisms are the \emph{reversed} natural transformations between the \emph{right} adjoint functors,
            \[
                \AdjR(F\dashv U, F'\dashv U') \coloneq \Nat(U', U), 
            \]
            composition of morphisms is reversed composition of the natural transformations.
        \end{itemize}
    \end{enumerate}
\end{definition}

\begin{prop}
    \label{prop:conjugation-properties}
    The conjugation maps defined above form an identity-on-objects isomorphism of categories
    \[
        \begin{tikzcd}[ampersand replacement=\&, column sep=tiny]
            {\AdjL} \& \cong \& {\AdjR.}
            \arrow["\jL", shift left, curve={height=-2pt}, from=1-1, to=1-3]
            \arrow["{\jR}", shift left, curve={height=-2pt}, from=1-3, to=1-1]
        \end{tikzcd}
    \]
\end{prop}
\begin{proof}
    We just need to show that \(\jL\) and \(\jR\) are functorial, the isomorphism part follows from the previous proposition.  It is clear from the zig-zag relations that they preserve identity morphisms, so it remains that they preserve composition.  We do this for \(\jL\), noting that the composition in the right-adjoint category \(\AdjR\) is reversed.
    \[
        \jL(\theta') \odot \jL(\theta)
        ~=~
        \raisebox{-.45\height}{\input{\figdir/conjugate_reverse_composition_1.pstex_t}}
        ~=~
        \raisebox{-.45\height}{\input{\figdir/conjugate_reverse_composition_2.pstex_t}}
        ~=~
        \raisebox{-.45\height}{\input{\figdir/conjugate_reverse_composition_3.pstex_t}}
        ~=~
        \jL(\theta \odot \theta')
    \]
\end{proof}
The following is an immmediate, almost trivial corollary, but it will be key later on.
\begin{cor}
    Given a pair of adjunctions \(F \dashv U\) and \(F' \dashv U'\) we have that
    \begin{gather*}
        F \xRightarrow{\theta} F'
        \quad\text{has inverse}\quad 
        F' \xRightarrow{\theta^{-1}} F\\
        \intertext{if and only if}        
        U' \xRightarrow{\jL(\theta)} U
        \quad\text{has inverse}\quad
        U \xRightarrow{\jL(\theta^{-1})} U'.
    \end{gather*}
\end{cor}

\section{Closed monoidal categories and extranatural transformations}
\label{sec:closed-monoidal}
This section is about the reformulation of the definition of closed monoidal category in terms of extranatural transformations.  This gives an object-free, or formal, approach which goes hand-in-hand with a surface diagrammatic representation.

\subsection{Closed monoidal categories}
\label{subsec:closed-monoidal-cats}

Here we will examine some standard definitions of closed monoidal category, setting the scene for Eilenberg and Kelly's notion of extranatural transformation which we introduce in the next subsection, and then a reformulation of the definition of closed monoidal category in Section~\ref{subsec:closed-mon-via-extranat}.  

A standard way~\cite{nlab:closed_monoidal_category} to define a closed monoidal category is as follows.
\begin{definition}
    \label{defn:closed-mon-cat-low-brow}
    A \define{left closed monoidal category} consists of a monoidal category $(\cC, \otimes, \one)$ together with, for each $c \in \cC$, a specified functor $[c, {-}]\colon \cC \to \cC$ and an adjunction $c\otimes {-} \dashv [c, {-}]$. The  counit and unit of each of these adjunctions are known as evaluation and coevaluation, respectively, with components are as follows:
    \[
        \ev_{c,\secondc} \colon c \otimes [c, \secondc] \to \secondc,
        \quad
        \coev_{c,\secondc} \colon \secondc \to [c, c \otimes \secondc].
    \]
    The object \([c, \secondc]\) is the \define{internal hom} from \(c\) to \(\secondc\).
\end{definition}
Two fundamental examples are, on one hand, the monoidal category $(\Set, \times, \{\star\})$ of sets equipped with the cartesian product as monoidal structure and the internal hom between two sets being the set of functions between the two sets, and, on the other hand, the monoidal category $(\Vect, \otimes, \C)$ of complex vector spaces with the internal hom being the set of linear functions equipped with the pointwise vector space structure.

The word `left' appears in the above, one obtains the `right' version by replacing $c\otimes {-}$ with ${-}\otimes c$.  In the context of symmetric monoidal categories then the left and right internal homs are canonically isomorphic, otherwise, however, if we are using both left and right versions then we should distinguish them for instance by using $\internalhoml{{-}}{{-}}$ and $\internalhomr{{-}}{{-}}$.  In some contexts, as we will see, it is convenient to use `infix' notation such as the `lollipop', so we write $\lollihoml{c}{\secondc} \equiv \internalhoml{c}{\secondc}$ and $\lollihomr{c}{\secondc} \equiv \internalhomr{c}{\secondc}$ with adjunctions $(c\otimes{-}) \dashv (\lollihoml{c}{{-}})$ and $({-} \otimes c) \dashv (\lollihomr{c}{-})$. 

If it helps to keep a non-symmetric example in mind, for any group \(\Gamma\) there is the monoidal category $\cC_\Gamma$ where the objects are the elements of \(\Gamma\), where the monoidal product is the product in the group and where left and right internal homs are given as follows: $\internalhoml{g}{h} = g^{-1}h$ and $\internalhomr{g}{h} = hg^{-1}$.

It is generally well-known that from the definition above one can canonically combine the family of functors $\bigl\{[c, {-}]\colon \cC\to \cC\bigr\}_c$ into a single functor $[{-}, {-}] \colon \cC^\op \times \cC \to \cC$.  We will go into the details.  
Begin by recalling that the evaluation and coevaluation morphisms are both natural in $Y$ and satisfy the following two triangle identities.
\begin{equation}
\begin{tikzcd}[ampersand replacement=\&,column sep=large]
	{[c,\secondc]} \& {\bigl[c,c\otimes [c,\secondc]\bigr]} \\
	\& {[c, \secondc]}
	\arrow["{[c,\ev_{c,\secondc}]}", from=1-2, to=2-2]
	\arrow["{\coev_{c, [c,\secondc]}}", from=1-1, to=1-2]
	\arrow["\id"', from=1-1, to=2-2]
\end{tikzcd}
    \quad
\begin{tikzcd}[ampersand replacement=\&,column sep=large]
	{c \otimes \secondc} \& {c\otimes[c, c\otimes \secondc]} \\
	\& {c \otimes \secondc}
	\arrow["{\ev_{c,c\otimes \secondc}}", from=1-2, to=2-2]
	\arrow["{c\otimes \coev_{c, \secondc}}", from=1-1, to=1-2]
	\arrow["\id"', from=1-1, to=2-2]
    \tag{$*$}
    \label{diag:triangle-ids}
\end{tikzcd}
\end{equation}
Then for $f\colon c\to c'$, define $[f, \secondc]\colon [c', \secondc] \to [c, \secondc]$ to be the lower composite of the following diagram, so that the diagram commutes by definition.
\[
\begin{tikzcd}[ampersand replacement=\&,column sep=large]
	{[c',\secondc]} \& {[c,\secondc]} \\
	{\bigl[c, c\otimes [c',\secondc]\bigr]} \& {\bigl[c, c'\otimes [c',\secondc]\bigr]}
	\arrow["{\coev_{c,[c',\secondc]}}"', from=1-1, to=2-1]
	\arrow["{[c, f\otimes [c',\secondc]]}", from=2-1, to=2-2]
	\arrow["{[f, \secondc]}", from=1-1, to=1-2]
	\arrow["{[c, \ev_{c',\secondc}]}"', from=2-2, to=1-2]
\end{tikzcd}
\tag{$**$}
\label{diag:functor-defn}
\]
(Another way to say it is that \([f, {-}]\) is defined as the conjugate of \(f \otimes {-} \).)
This gives the requisite \define{left internal hom functor} $[{-}, {-}] \colon \cC^\op \times \cC \to \cC$.

Despite being natural in $\secondc$, the families of maps $\{\coev_{c,\secondc}\}_{c,\secondc}$ and $\{\ev_{c,\secondc}\}_{c,\secondc}$ do not form natural transformations, indeed the domains and codomains do not match up properly, with a repeated variable of $c$ on one side.  In fact these families satisfy what is known as extranaturality, introduced by Eilenberg and Kelly~\cite{EilenbergKelly}, we will consider this concept in more generality in the next subsection.  For completeness, the proofs of the next two theorems are in Appendix~\ref{app:proofs}.

\begin{thm}
    \label{thm:extranatural-coev-ev}
    In a left closed monoidal category 
    the evaluation and coevaluation morphisms are `extranatural'  in the first parameter, $c$, in the sense that for any morphism $f\colon c\to c'$ the following diagrams commute. 
    \[
\begin{tikzcd}[ampersand replacement=\&]
	{c\otimes [ c', \secondc]} \& {c'\otimes [c', \secondc]} \\
	{c\otimes [c, \secondc]} \& \secondc
	\arrow["{c \otimes [f,  \secondc]}", tail reversed, no head, from=2-1, to=1-1]
	\arrow["{ f\otimes [c',\secondc]}"', tail reversed, no head, from=1-2, to=1-1]
	\arrow["{\ev_{c', \secondc}}"', tail reversed, no head, from=2-2, to=1-2]
	\arrow["{\ev_{c,\secondc}}", tail reversed, no head, from=2-2, to=2-1]
\end{tikzcd}
        \quad
\begin{tikzcd}[ampersand replacement=\&,column sep=large]
	\secondc \& {[c', c'\otimes \secondc]} \\
	{[c, c\otimes \secondc]} \& {[c, c'\otimes \secondc]}
	\arrow["{[f, c'\otimes \secondc]}", from=1-2, to=2-2]
	\arrow["{[c, f\otimes \secondc]}"', from=2-1, to=2-2]
	\arrow["{\coev_{c, \secondc}}"', from=1-1, to=2-1]
	\arrow["{\coev_{c',\secondc}}", from=1-1, to=1-2]
\end{tikzcd}
    \]
\end{thm}
Now we can give  an equivalent definition of a left closed monoidal category.
\begin{thm}
    \label{thm:closed-mon-cat-med-brow}
    A left closed monoidal category in the sense of Definitions~\ref{defn:closed-mon-cat-low-brow} can equivalently be specified in the following way: a monoidal category $(\cC, \otimes, \one)$ with a functor $[{-}, {-}]\colon \cC^\op \times \cC \to \cC$ and, for all $c,\secondc\in \cC$, morphisms 
    \[
        \ev_{c,\secondc} \colon c \otimes [c, \secondc] \to \secondc
        \quad
        \text{and}
        \quad
        \coev_{c,\secondc} \colon \secondc \to [c, c \otimes \secondc]
    \]
    which are natural in $\secondc$ and extranatural in $c$ (in the sense of Theorem~\ref{thm:extranatural-coev-ev}) and which satisfy the triangle identities of~(\ref{diag:triangle-ids}) above.
\end{thm}

In order to reformulate the definition of monoidal category from Theorem~\ref{thm:closed-mon-cat-med-brow} into the form we want it (Theorem~\ref{thm:closed-mon-cat-high-brow}), we need to look at the general concept of extranatural transformation.

\subsection{Extranatural transformations}
\label{subsec:extranatural-transformations}
In Theorem~\ref{thm:extranatural-coev-ev} we saw that evaluation and coevaluation satisfied a condition which we referred to as being extranatural.  In this subsection we will consider the general notion of extranatural transformation first defined by Eilenberg and Kelly~\cite{EilenbergKelly}.  We will give the diagrammatic representation and give some basic properties.  Extranatural transformations will be fundamental for the definition we want of closed monoidal category and, more generally, of adjunctions of two variables.  This section is a mild updating of material from my work of Hopf monads~\cite{Willerton:HopfMonads} where I used the term `dinatural' instead.

Let's begin with the definition.
\begin{definition}
    \label{defn:extranaturality}
    Suppose
    \[
        P\colon \cC\times \cC^\op\times \cA\to \cD
        \quad\text{and}\quad
        Q\colon \cA\times \cB^\op\times \cB \to \cD
    \]
    are two functors.  
    An \define{extranatural transformation} 
    \(\beta\colon P \extranat Q\) is a family of morphisms
    \(\beta_{c,a,b} \colon P(c,c,a)\to Q(a,b,b)\) which satisfies
    the following extranaturality condition.  If $f\colon c\to c'$, $h\colon a \to a'$ and $g\colon b\to b'$  are morphisms in, respectively, \(\cC\), \(\cA\) and \(\cB\) then the diagram below commutes.
    \[
\begin{tikzcd}[ampersand replacement=\&]
	{P(c, c', a)} \& {P(c',c',a)} \& {P(c', c', a')} \& {Q(a', b', b')} \\
	{P(c, c, a)} \& {Q(a, b, b)} \& {Q(a', b, b)} \& {Q(a', b, b')}
	\arrow["{P(f, c', a)}", from=1-1, to=1-2]
	\arrow["{P(c, f, a)}"', from=1-1, to=2-1]
	\arrow["{P(c',c', h)}", from=1-2, to=1-3]
	\arrow["{\beta_{c', a, b}}"', from=1-2, to=2-2]
	\arrow["{\beta_{c', a', b'}}", from=1-3, to=1-4]
	\arrow["{\beta_{c', a', b}}", from=1-3, to=2-3]
	\arrow["{Q(a', g, b')}", from=1-4, to=2-4]
	\arrow["{\beta_{c,a,b}}"', from=2-1, to=2-2]
	\arrow["{Q(h,b,b)}"', from=2-2, to=2-3]
	\arrow["{Q(a', b, g)}"', from=2-3, to=2-4]
\end{tikzcd}
    \]
\end{definition}

The first square expresses extranaturality in $c$, the second square expresses ordinary naturality in $a$ and the final square expresses extranaturality in $b$.  In Appendix~\ref{app:double-category}, Theorem~\ref{thm:extranatural-as-natural} gives an alternative definition of extranatural transformation as a natural transformation between profunctors.

Each of the categories  \(\cC\), \(\cA\), \(\cB\) and \(\cD\) can, of course, be products of
other categories and these can be permuted in the definition.  A usual natural transformation is the case where
\(\cC = \cB = \termcat\).

\begin{ex}
    For $\cC$ a left closed monoidal category as per Definition~\ref{defn:closed-mon-cat-low-brow}, if we take $P\colon \cC \times \cC^\op \times \cC \to \cC$ to be given by $P(c, c', a) \coloneq c \otimes [c', a]$ and $Q\colon \cC \to \cC$ to be the identity functor, then Theorem~\ref{thm:extranatural-coev-ev} says that the evaluation maps form an extranatural transformation $\ev\colon P \extranat Q$.   
    Similarly, Theorem~\ref{thm:extranatural-coev-ev} also says that the coevaluation maps form an extranatural transformation, $\coev$.  
\end{ex}

A key point is that we can represent extranatural transformations diagrammatically.  If we have functors
\[
    P\colon \cC\times \cC^\op\times \cA\to \cD
    \quad\text{and}\quad
    Q\colon \cA\times \cB^{\op}\times \cB \to \cD
\]
then an extranatural transformation \(\beta\colon
P\extranat Q\) can be denoted as below, where we are using the conventions from Section~\ref{subsec:surface-diagrams} so the diagram is read bottom-to-top, right-to-left and front-to-back, and any part of the surface diagram that represents the opposite $\cC^\op$ of a category $\cC$ is shaded, giving the sense of two sides to the surface.  
\[
   \raisebox{-.45\height}{\input{\figdir/extranatural_sample.pstex_t}}
\]
This is not a natural transformation; it is not a \(2\)-cell the monoidal \(2\)-category of categories.  However, in  Appendix~\ref{app:double-category} we see it can be viewed as a \(2\)-cell in the monoidal double category of profunctors.

We can horizontally compose extranatural transformations with functors, as we do ordinary natural transformations, where it is also called `whiskering'.  Suppose that \(F\colon \tilde \cA\to \cA\),  \(G\colon \tilde \cC\to
\cC\),  \(H\colon \tilde \cB\to \cB\), and  \(K\colon \cD\to
\tilde\cD\) are functors  and that \(\beta \colon P\extranat Q\) is an
extranatural transformation of the above form then we get an extranatural
transformation
 \[\tilde \beta\colon K\circ P\circ (G\times G^\op\times F)\extranat
     K\circ Q\circ (F\times H^\op\times H),\]
given by
\(\tilde\beta_{\tilde c, \tilde a, \tilde b}:=K\bigl(\beta_{G(\tilde c),
  F(\tilde a), H(\tilde b)}\bigr)\).
This is denoted graphically as follows.
 \[\raisebox{-.45\height}{\input{\figdir/whiskered_extranat_sample.pstex_t}}\]

We can also compose extranatural transformations with ordinary natural transformations.
Suppose that we have
functors 
\[
    F, F' \colon \tilde\cA \to \cA,\quad 
    G, G' \colon \tilde\cC \to \cC, \quad 
    H, H' \colon \tilde\cB \to \cB,\quad
    K, K' \colon \cD \to \tilde\cD,
\] 
and natural transformations
\[
    \phi \colon F \nattrans F',\quad 
    \gamma \colon G\nattrans G', \quad
    \theta\colon H' \nattrans H, \quad
    \kappa\colon K \nattrans K',
\]
together with
an extranatural transformation
\(\beta \colon P\extranat Q\)
 of the above form
then we have the following composite extranatural
transformations and the equality between them.  The pictures should demonstrate a key feature of the diagrammatic notation.
 \[
    \raisebox{-.45\height}{\input{\figdir/nt_composed_extranat_1.pstex_t}}\quad
    =
    \quad
    \raisebox{-.45\height}{\input{\figdir/nt_composed_extranat_2.pstex_t}}.
\]
The equality follows from the naturality of $\kappa$ together with the extranaturality of $\beta$ as expressed by the commutativity of the diagram in Definition~\ref{defn:extranaturality}.  

We can also vertically compose two extranatural transformations to obtain another extranatural transformation provided a certain condition pointed out by 
Eilenberg and Kelly~\cite{EilenbergKelly} is satisfied.  The right-hand profile of the surface we use to represent an extranatural transformation is the so-called \define{Eilenberg-Kelly graph} of the extranatural
transformation, consisting of arcs with the end-points of an arc
labelled by the same category.  We can define a composite extranatural transformation provided the composite Eilenberg-Kelly graph has no loops.  This is probably best illustrated with some examples.  Suppose we have
functors
\[
    P\colon \cC\times \cC^\op\times \cA\to \cD,\quad
    R\colon \cA\times \cA^\op\times \cA\to \cD, \quad 
    Q\colon \cA\times \cB\times \cB^\op \to \cD
\]
together with extranatural transformations \(\beta \colon P\extranat R\) and
\(\beta' \colon R\extranat Q\) such that $\beta$ is extranatural in the last two factors of the domain of $R$.  We can form the composite extranatural transformation provided that $\beta'$ is extranatural in the \emph{first} two factors of the domain of $R$ but not if $\beta'$ is extranatural in the \emph{last} two factors of the domain of $R$.  This is illustrated in Figure~\ref{fig:eilenberg-kelly-graph}.  The left-hand picture gives an extranatural transformation, but the right-hand picture does not.  However, using the ideas of  Appendix~\ref{app:double-category}, the right-hand side picture does represent a natural transformation between profunctors.
\begin{figure}[tbh]
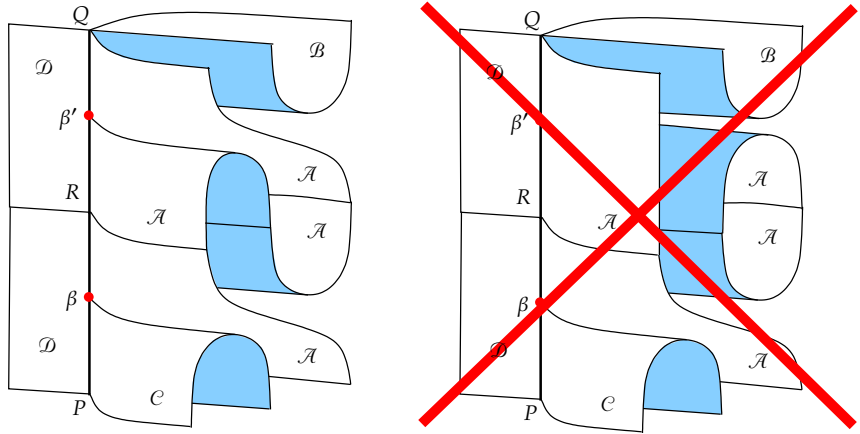

\[
    \raisebox{-.45\height}{\input{\figdir/extranatural_vertical_composite.pstex_t}}
    \qquad
    \raisebox{-.45\height}{\input{\figdir/extranatural_vertical_non_composite.pstex_t}}
\]
    \caption{Example and non-example of composition of extranatural transformations}
    \label{fig:eilenberg-kelly-graph}
\end{figure}

\subsection{Closed monoidal categories via extranatural transformations}
\label{subsec:closed-mon-via-extranat}
We can now give the final reformulation of the definition of closed monoidal category.  This is amenable to representation in the diagrammatic framework.  We will see later that this is an example of a reformulation of the notion of an adjunction of two variables.  This will be a useful perspective for understanding key results of Fausk, Hu and May, below.

Here it will be convenient to switch to the lollipop notation $c\leftlolli c'$ for the internal hom which we called $\internalhoml{c}{c'}$ above.

\begin{thm}
    \label{thm:closed-mon-cat-high-brow}
    A {left closed monoidal category} can equivalently defined as a monoidal category $(\cC, \otimes, \one)$ with a functor ${\leftlolli} \colon \cC^\op \times \cC \to \cC$ and two extranatural transformations (natural in the third variable in both cases),
    \[
        \ev \colon {\otimes} \circ ({\id} \times {\leftlolli}) \extranat {\id}
        \quad
        \text{and}
        \quad
        \coev \colon {\id} \extranat {\leftlolli} \circ ({\id} \times {\otimes}),
    \]
    satisfying the triangle identities in the sense that the following diagrams commute.
    \begin{gather*}
        \begin{tikzcd}[ampersand replacement=\&,sep=2.25em]
            {\leftlolli } \& {{\leftlolli}\circ ({\id} \times {\otimes}) \circ ({\id} \times {\id} \times {\leftlolli})} \\
            \& \leftlolli
            \arrow["\id"', Rightarrow, from=1-1, to=2-2]
            \arrow["{{\leftlolli} \circ (\id \times \ev)}", "\shortmid"{marking},
            Rightarrow, from=1-2, to=2-2]
            \arrow["{\coev \circ {\leftlolli}}", "\shortmid"{marking}, Rightarrow, from=1-1, to=1-2]
        \end{tikzcd}
        \\
        \begin{tikzcd}[ampersand replacement=\&,sep=2.25em]
            \otimes \& {{\otimes}\circ ({\id} \times {\leftlolli}) \circ ({\id} \times {\id} \times {\otimes})} \\
            \& \otimes
            \arrow["\id"', Rightarrow, from=1-1, to=2-2]
            \arrow["{{\ev} \circ ({\id} \times {\id} \times {\otimes})}", "\shortmid"{marking}, Rightarrow, from=1-2, to=2-2]
            \arrow["{{\otimes} \circ {\coev} }", "\shortmid"{marking}, Rightarrow, from=1-1, to=1-2]
        \end{tikzcd}
    \end{gather*}
\end{thm}
We represent the evaluation and coevaluation extranatural transformations as follows.
\[
    \raisebox{-.45\height}{\input{\figdir/ev.pstex_t}}
    \qquad
    \raisebox{-.45\height}{\input{\figdir/coev.pstex_t}}
\]
The triangle identities are then expressed in the following way.
\begin{align*}
    \raisebox{-.45\height}{\input{\figdir/triangle_1_lhs.pstex_t}}
    ~&=~
    \raisebox{-.45\height}{\input{\figdir/triangle_1_rhs.pstex_t}}\\
    \raisebox{-.45\height}{\input{\figdir/triangle_2_lhs.pstex_t}}
    ~&=~
    \raisebox{-.45\height}{\input{\figdir/triangle_2_rhs.pstex_t}}
\end{align*}
We will abstract the key parts of the above definition now, in to the notion of an adjunction of two variables which will allow us to handle the examples we are interested in.

\section{Adjunctions of two variables and conjugation}
\label{sec:adjunctions-of-two-variables}

This section is in some sense the heart of the paper.  We give an extranatural definition of adjunction of two variables which is a unit-counit reformulation in the formal setting of Kan's notion of a pair of adjoint functors in two variables. 
We see how conjugation, generalized from ordinary adjunctions, works in this setting, seeing analogues of the results in Section~\ref{sec:adjunctions-and-conjugation}.  We then look at conjugation for certain composite adjunctions of two variables, the construction here being key for the examples.

\subsection{Defintion of adjunctions of two variables}
The relationship between tensor and internal hom in Theorem~\ref{thm:closed-mon-cat-high-brow} above is a special case of the following.
\begin{definition}
    A \define{left-sided adjunction of two-variables} consists of two functors $T \colon \cA \times \cB \to \cC$ and $H \colon \cA^\op \times \cC \to \cB$ together with two extranatural transformations, the counit and the unit,
    \[
        \epsilon \colon T \circ ({\id_\cA} \times H) \extranat \id_\cC
        \quad\text{and}\quad
        \eta \colon \id_\cB \extranat H \circ ({\id_{\cA^\op}} \times T),
    \]
    satisfying the triangle identities in the sense that the following diagrams commute.
    \begin{gather*}
    \begin{tikzcd}[ampersand replacement=\&]
        {H} \& {{H}\circ ({\id} \times {T}) \circ ({\id} \times {\id} \times {H})} \\
        \& H
        \arrow["\id"', Rightarrow, from=1-1, to=2-2]
        \arrow["{{H} \circ ({\id} \times {\epsilon})}", "\shortmid"{marking}, Rightarrow, from=1-2, to=2-2]
        \arrow["{\eta \circ {H}}", "\shortmid"{marking}, Rightarrow, from=1-1, to=1-2]
    \end{tikzcd}\\
    \begin{tikzcd}[ampersand replacement=\&]
        T \& {{T}\circ ({\id} \times {H}) \circ ({\id} \times {\id} \times {T})} \\
        \& T
        \arrow["\id"', Rightarrow, from=1-1, to=2-2]
        \arrow["{{\epsilon} \circ ({\id} \times {\id} \times {T})}", "\shortmid"{marking}, Rightarrow, from=1-2, to=2-2]
        \arrow["{{T} \circ {\eta} }", "\shortmid"{marking}, Rightarrow, from=1-1, to=1-2]
    \end{tikzcd}
\end{gather*}
    Say that $T$ is the left adjoint and $H$ is the right adjoint and write $T\leftsidedadj H$.
\end{definition}
So a left closed monoidal category is a monoidal category $(\cC, \otimes, \one)$ with a functor $\leftlolli\colon \cC^\op \times \cC \to \cC$ and a left-sided adjunction ${\otimes} \leftsidedadj {\leftlolli}$.  Diagrammatically, we represent a left-sided adjunction in the same way we have represented a left closed monoidal category above, but of course the labels change on the pictures.

This definition of left-sided adjunction of two-variables is, up to permutation, a unit-counit reformulation of Kan's notion~\cite{KanAdjointFunctors} of `adjoint functors in two variables' (the term `adjunction with parameter' is used in Mac Lane's book~\cite{MacLaneCategoriesWork}).  In Kan's definition $T$ would have domain $\cB \times \cA$ rather than $\cA \times \cB$; in the symmetric monoidal $2$-category of categories this doesn't really matter, but in the diagrammatic world, our convention is sensible as we wish to avoid using symmetries without good reason.  Kan's convention makes sense when defining adjoint functors in two variable via a natural isomorphism $\Hom_\cC(T({-}, {-}), {-})\cong \Hom_\cB({-}, H({-}, {-}))$.

The definition of a \emph{right} closed monoidal category will clearly be a slight modification of the above.  A right closed monoidal category is a monoidal category $(\cC, \otimes, \one)$ with a functor $\rightlolli\colon \cC \times \cC^\op \to \cC$ and a right-sided adjunction of two variables ${\otimes} \rightsidedadj {\rightlolli}$, where a right-sided adjunction of two variables involves functors $T \colon \cA \times \cB \to \cC$ and $H \colon \cC \times \cB^\op \to \cA$ together with two extranatural transformations, the counit and the unit,
$
    \epsilon \colon T \circ (H \times {\id_\cB}) \extranat \id_\cC
$
and
$
    \eta \colon \id_\cA \extranat H \circ (T \times {\id_{\cB^\op}}),
$
satisfying appropriate triangle identities.  The diagrams representing the structure here are just the diagrams for the left-sided version but reflected in the plane of the paper or screen.

It is worth noting that in Hovey's book~\cite{HoveyModelCategories} an adjunction of two variables means a functor $T\colon \cA \times \cB \to \cC$ with \emph{both} left-sided and right-sided adjunctions of two variables $T\leftsidedadj H$ and $T\rightsidedadj H'$, so is modelling the situation where a monoidal category has both left and right internal homs.  Cheng, Gurski and Riehl~\cite{ChengGurskiRiehl} give a more symmetric definition of this notion of having both a left-sided and a right-sided adjunction of two variables.

\subsection{New adjunctions from old}  
\label{subsec:new-adjunctions-from-old}
We can compose ordinary adjunctions to obtain new adjunctions.  We can similarly compose left-sided adjunctions of two variables: we will not look at the fully general case, but rather look at the following special case that will give us the examples we are interested in.  
Suppose that we have a left-sided adjunction of two variables \(T \leftsidedadj H\) for the functors
\[ 
    T \colon \cA \times \cB \to \cC
    \quad\text{and}\quad
    H \colon \cA^\op \times \cC \to \cB, 
\]
drawn as 
\[
    \raisebox{-.45\height}{\input{\figdir/left_sided_left_adjoint.pstex_t}}
    \quad\text{and}\quad
    \raisebox{-.45\height}{\input{\figdir/left_sided_right_adjoint.pstex_t}}
\]
and that we also have a functor
\[ 
    K \colon \tilde A \to A
\]
together with a pair of ordinary adjunctions
\[
    F_1 \colon \cC \leftrightarrows \tilde \cC \cocolon U_1
    \quad\text{and}\quad
    F_2 \colon \tilde \cB \leftrightarrows \cB \cocolon U_2
\]
then we get another left-sided adjunction of two variables:
\[
    F_1 \circ T \circ (K \times F_2) \leftsidedadj 
    U_2 \circ H \circ (K^{\op} \times U_1).
\]
The left and right functors in this adjunction of two variables are denoted as follows:
\[
    \raisebox{-.45\height}{\input{\figdir/left_sided_left_adjoint_modified.pstex_t}}
    \quad\text{and}\quad
    \raisebox{-.45\height}{\input{\figdir/left_sided_right_adjoint_modified.pstex_t}}
    \, .
\] 
The unit and counit of the  adjunction of two variables are simple to write down diagrammatically and they are as follows:
\[
    \raisebox{-.45\height}{\input{\figdir/coev_modified.pstex_t}}
    \quad\text{and}\quad
    \raisebox{-.45\height}{\input{\figdir/ev_modified.pstex_t}}
\]
One can check diagrammatically, very easily, that these satisfy the zig-zag identities.


\subsection{General conjugation for adjunctions of two variables}
\label{subsec:general-conjucation-two-variables}
For a pair of left-sided adjunctions of two variables we can define conjugation between sets of natural transformations analogously to how it is done for ordinary pairs of adjunctions as in Definition~\ref{defn:ordinary_conjugation}.
\begin{definition}
    \label{def:jL-jR-two-var}
    For a pair of left-sided adjunctions of two variables, \(T \leftsidedadj H\) and \(T' \leftsidedadj H'\) we can define \define{conjugation}, this consists of a pair of maps between sets of natural transformations
    \[
        \begin{tikzcd}[ampersand replacement=\&, column sep=tiny]
            {\Nat(T, T')} \&  \& {\Nat(H', H).}
            \arrow["\jL", shift left, curve={height=-1pt}, from=1-1, to=1-3]
            \arrow["{\jR}", shift left, curve={height=-1pt}, from=1-3, to=1-1]
        \end{tikzcd}
    \]
    Diagrammatically, \(\jL\) and \(\jR\) are defined as in Figure~\ref{fig:jL-jR-two-var}.
    \begin{figure}[tbh]
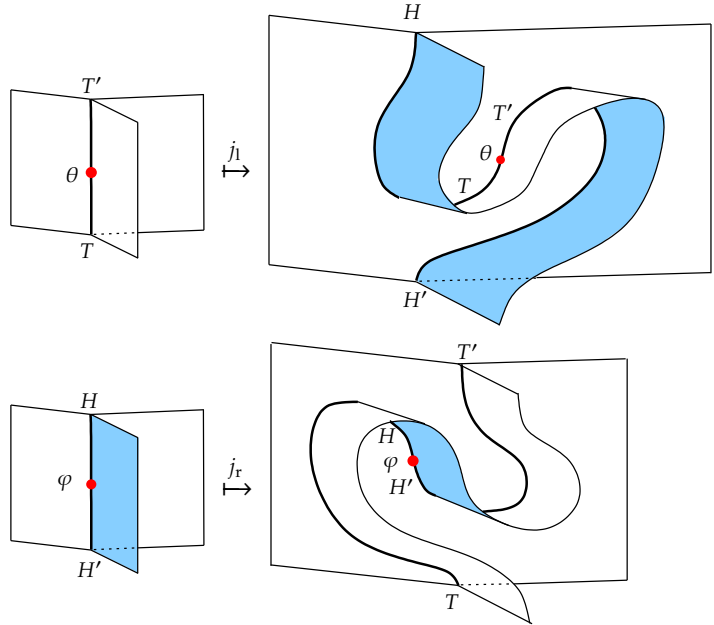

        \begin{align*}
            \raisebox{-.45\height}{\input{\figdir/two_var_conjugate_forward_pre.pstex_t}}
            ~&\xmapsto{\jL}~
            \raisebox{-.45\height}{\input{\figdir/two_var_conjugate_forward.pstex_t}}
            \\
            \raisebox{-.45\height}{\input{\figdir/two_var_conjugate_backward_pre.pstex_t}}
            ~&\xmapsto{\jR}~
            \raisebox{-.45\height}{\input{\figdir/two_var_conjugate_backward.pstex_t}}
            \end{align*}
        \caption{The functions \(\jL\) and \(\jR\) of Definition~\ref{def:jL-jR-two-var}}
        \label{fig:jL-jR-two-var}
    \end{figure}
\end{definition}
\begin{prop}
    The two conjugation maps defined above are mutually inverse.
\end{prop}
\begin{proof}
    This is proved diagrammatically in the same way as Proposition~\ref{prop:ordinary-conjugation-inverse}.
\end{proof}
We can also analogous categories of adjunctions of two variables.
\begin{definition}
    There is a category \(\TwadjLL\), the \define{left category of left-side adjunctions of two variables}, where an object is a left-sided adjunction of two variables \(T \leftsidedadj H\) and the hom-sets are given by 
    \[
        \TwadjLL(T \leftsidedadj H, T' \leftsidedadj H') := \Nat(T, T').
    \]

    Similarly, there is a category \(\TwadjLR\), the \define{right category of left-side adjunctions of two variables}, where an object is a left-sided adjunction of two variables \(T \leftsidedadj H\) and the hom-sets are given by 
    \[
        \TwadjLR(T \leftsidedadj H, T' \leftsidedadj H') := \Nat(H', H).
    \]  
    Composition is given by reverse composition of natural transformations.
\end{definition}
\begin{prop}
    \label{prop:two-variable-conjugation-properties}
    The conjugation maps defined above form an identity-on-objects isomorphism of categories
    \[
        \begin{tikzcd}[ampersand replacement=\&, column sep=tiny]
            {\TwadjLL} \& \cong \& {\TwadjLR.}
            \arrow["\jL", shift left, curve={height=-2pt}, from=1-1, to=1-3]
            \arrow["{\jR}", shift left, curve={height=-2pt}, from=1-3, to=1-1]
        \end{tikzcd}
    \]
\end{prop}
\begin{proof}
    This is proved diagrammatically in the same way as Proposition~\ref{prop:conjugation-properties}.
\end{proof}

The following is an immmediate, almost trivial corollary, but it will be key later on.
\begin{cor}
    \label{cor:two-conjugate-inverse}
    Given a pair of left-sided adjunctions of two variables, \(T \leftsidedadj H\) and \(T' \leftsidedadj H'\) we have that
    \begin{gather*}
        T \xRightarrow{\theta} T'
        \quad\text{has inverse}\quad 
        T' \xRightarrow{\theta^{-1}} T\\
        \intertext{if and only if}
        H' \xRightarrow{\jL(\theta)} H
        \quad\text{has inverse}\quad
        H \xRightarrow{\jL(\theta^{-1})} H'.
    \end{gather*}
\end{cor}

\subsection{Conjugation for composite adjunctions}
\label{subsec:composite-adj-conjugation}

This is the key general construction of the paper.

Suppose we have closed monoidal categories $\cC$ and $\cD$ together with four pairs of adjoint functors of the following form
\begin{align*}
    F_1 \colon & \cC \rightleftarrows \cE \colon U_1,
    &
    F_2 \colon & \cB \rightleftarrows \cC \colon U_2,
    &
    F'_1 \colon & \cD \rightleftarrows \cE \colon U'_1,
    &
    F'_2 \colon & \cB \rightleftarrows \cD \colon U'_2,
\end{align*}
together with two further functors
\[
    K\colon \cA \to \cC,\quad
    K'\colon \cA \to \cD,
\]
then, then by the construction in Section~\ref{subsec:new-adjunctions-from-old}, we have a pair of left-sided adjunctions of two variables:
\begin{gather*}
    F_1 \circ {\otimes_\cC} \circ (K \times F_2)
    \leftsidedadj
    U_2 \circ {\leftlolli_\cC} \circ (K \times U_1);\\
    F'_1 \circ {\otimes_\cD} \circ (K' \times F'_2)
    \leftsidedadj
    U'_2 \circ {\leftlolli_\cD} \circ (K' \times U'_1).
\end{gather*}
By the Proposition~\ref{prop:two-variable-conjugation-properties}, conjugation then gives a bijection between the sets of natural transformations
\begin{multline*}
    \bigl\{\theta\colon F_1 \circ {\otimes} \circ (K \times F_2)
    \nattrans
    F'_1 \circ {\otimes} \circ (K' \times F'_2)\bigr\}
    \\
    \cong
    \bigl\{\phi \colon U'_2 \circ {\leftlolli} \circ (K' \times U'_1)
    \nattrans
    U_2 \circ {\leftlolli} \circ (K \times U_1)\bigr\}  
    \tag{$\dagger$}
    \label{eq:conjugation-bijection}
\end{multline*}
that is, a natural transformation with components of the form 
\[
    \theta_{a,b} \colon F_1 \bigl(K(a) \otimes F_2(b)\bigr) \to F'_1 \bigl(K'(a) \otimes F'_2(b)\bigr)
    \tag{$\dagger.1$}
\]
is conjugate to one with components of the form
\[
    \phi_{a,e} \colon U'_2 \bigl(K'(a) \leftlolli U'_1(e)\bigr) \to U_2 \bigl(K(e) \leftlolli U_1(e)\bigr).
    \tag{$\dagger.2$}
\]
We will call these \define{conjugate shapes}.

Conjugation is given in the following way, where the colours on the functors are just there to help the reader. 
\[
    \raisebox{-.45\height}{\input{\figdir/conjugate_forward_pre.pstex_t}}
    ~\mapsto~
    \raisebox{-.45\height}{\input{\figdir/conjugate_forward.pstex_t}}
\]
\[
    \raisebox{-.45\height}{\input{\figdir/conjugate_backward.pstex_t}}
    ~\mapsfrom~
    \raisebox{-.45\height}{\input{\figdir/conjugate_backward_pre.pstex_t}}
\]
Writing this down in components is not particularly easy, as the expressions get rather large.

Let us now turn to applying this construction.

\section{Examples}
\label{sec:examples}

We can now apply the general construction of Section~\ref{subsec:composite-adj-conjugation} to the examples we are interested in, these are essentially the examples of Fausk, Hu and May.  We will first look at conjugate shapes from common strings of adjunctions and then focus on a certain internal adjunction and on the projection formula.

\subsection{Examples of conjugate shapes from strings of adjunctions}
Let's now look at some examples, which are essentially the examples in the paper of Fausk, Hu and May~\cite{FauskHuMay:Isomorphisms}.

Suppose that we have left closed monoidal categories $\cX$ and $\cY$ together with a functor $f^\ast \colon \cY \to \cX$.  The existence of various strings of adjunctions which include \(f^\ast\) will allow us to write down various conjugate shapes as in the previous section.  We will start with an example where it is spelt out.

Let's start with the basic example of having an adjunction \(f^\ast\dashv f_\ast\).  Then in the construction of Section~\ref{subsec:composite-adj-conjugation} we can take the following four adjunctions
\begin{align*}
    \id \colon & \cX \rightleftarrows \cX \colon \id,
    &
    f^\ast \colon & \cY \rightleftarrows \cX \colon f_\ast,
    &
    f^\ast \colon & \cY \rightleftarrows \cX \colon f_\ast,
    &
    \id \colon & \cY \rightleftarrows \cY \colon \id,
\end{align*}
together with the two functors
\[
    f^*\colon \cY \to \cX,\quad
    \id\colon \cY \to \cY.
\]
The result of the construction is that conjugation gives a bijection between natural transformations with components of shape
\[
    f^\ast(y) \otimes f^\ast(y') \to  f^\ast(y \otimes y')
\]
and natural transformations with components shape
\[
    y \leftlolli f_*(x) \to f_*(f^\ast(y) \leftlolli x).
\]
A diagram of going one way is given in Section~\ref{subsec:internal-adjunction} below.

With other, similar strings of adjunctions, we can, similarly, form other composite adjoints to use in the construction of Section~\ref{subsec:composite-adj-conjugation}.  Here are some examples, essentially those from Fausk, Hu and May~\cite{FauskHuMay:Isomorphisms}, starting with the one we've just done.  (Note unlike Fausk, Hu and May we are not assuming the monoidal structures are symmetric which clarifies the cases a little, we will add in the symmetry below.)  In each case, conjugation gives a bijection between natural transformations whose components have the shape of the morphism on the left and natural transformations whose components have the shape of the morphism on the right.  By comparing these to the general form \eqref{eq:conjugation-bijection} in Section~\ref{subsec:composite-adj-conjugation} it should be clear which four adjunctions and two functors have been used in each case.
\begin{align}
    \intertext{$f^\ast\dashv f_\ast$:}
    f^\ast(y) \otimes f^\ast(y') &\to  f^\ast(y \otimes y');
    &
    y \leftlolli f_*(x) &\to f_*(f^\ast(y) \leftlolli x).
    \label{eq:first-conjs}
    \\[0.5em]
    \intertext{$f_!\dashv f^\ast$:}
    f_!(f^\ast(y) \otimes x) &\to y \otimes f_!(x);
    &
    f^\ast(y \leftlolli y') &\to f^\ast(y) \leftlolli f^\ast(y').
    \label{eq:second-conjs}
    \\[0.5em]
    \intertext{$f_!\dashv f^\ast\dashv f_\ast$:}
    f_!(x \otimes f^\ast(y))&\to f_!(x) \otimes y;
    &
    f_!(x) \leftlolli y &\to f_\ast(x \leftlolli f^\ast(y)).
    \label{eq:third-conjs}
    \\[0.5em]
    \intertext{$f^\ast\dashv f_\ast \dashv f^!$:}
    f_*(x \otimes f^\ast(y)) &\to f_*(x) \otimes y;
    &
    f_*(x) \leftlolli y &\to f_*(x \leftlolli f^!(y)).
    \label{eq:fourth-conjs}
    \\[0.5em]
    \intertext{$f^\ast\dashv f_\ast \dashv f^!$:}
    f_*(f^\ast(y) \otimes x) &\to y \otimes f_*(x);
    &
    f^!(y \leftlolli y') &\to f^\ast(y) \leftlolli f^!(y').
    \label{eq:fifth-conjs}
\end{align}
Note that it is the same string of adjunctions used in  (\ref{eq:fourth-conjs}) and (\ref{eq:fifth-conjs}), however, the various functors are used in different places.  We think of the string of adjunctions in (\ref{eq:third-conjs}) and (\ref{eq:fourth-conjs}) as being different as we are thinking of \(f^\ast\) as a distinguished functor, because in practice it will be monoidal.

There is nothing special about the choice of the morphism directions.  We get similar conjugation bijections when we switch the direction on both sides.

We can observe that if $\cX$ and $\cY$ are \emph{braided} monoidal categories -- in particular, \emph{symmetric} -- then a natural transformation with the form of the left-hand side of (\ref{eq:second-conjs}) is canonically isomorphic to a natural transformation with the form of the left-hand side of (\ref{eq:third-conjs}).  Thus, thinking of these two natural transformations as being essentially the same, in this braided monoidal case we can think of a string of adjunctions \(f_!\dashv f^\ast\dashv f_\ast\) as giving rise to `conjugate triads' of natural transformations, namely the left hand side of \eqref{eq:second-conjs} with the right hand sides of \eqref{eq:second-conjs} and \eqref{eq:third-conjs}.  In the same way, in the braided monoidal case,  (\ref{eq:fourth-conjs}) and (\ref{eq:fifth-conjs}) give rise to `conjugate triads' of natural transformations.  These are the conjugate triads of Fausk, Hu and May~\cite[Section~4]{FauskHuMay:Isomorphisms}.

\subsection{An internal adjunction for strong monoidal functors}
\label{subsec:internal-adjunction}

Having an adjunction \(F\colon  \cC \rightleftarrows \cD \cocolon U\) means having an isomorphism of hom-sets \(\cD(F(C), D) \cong \cC(C, U(D))\).  Having an `internal adjunction' means using internal homs rather than hom-sets, but we can't generally posit an isomorphism between the internal homs \(F(C) \leftlolli_{\cD} D\) and \(C \leftlolli_{\cC} U(D)\) because they generally live in different categories.  This means that we need to use some functor to compare them, for example we can compare  \(F(C) \leftlolli_{\cD} D\) and \(F(C \leftlolli_{\cC} U(D))\).   

Now if we have a monoidal functor \(f^\ast \colon \cY \to \cX\) then we have the structural natural transformation \(\monstr{f^\ast}\) of shape \(f^\ast(y) \otimes f^\ast(y') \to f^\ast(y \otimes y')\) that is, a natural transformation \(\monstr{f^\ast} \colon {\otimes} \circ (f^\ast \times f^\ast) \nattrans f^\ast \circ {\otimes}\) which is drawn as follows.
\[
    \raisebox{-.45\height}{\input{\figdir/monoidal_structure.pstex_t}}
\]
This will satisfy appropriate axioms, but we are not interested in those here.  The functor \(f^\ast\) is strong monoidal precisely when the natural transformation \(\monstr{f^\ast}\) is invertible.

If the functor \(f^\ast\) has a right adjoint \(f_\ast\) then we are in the situation of \eqref{eq:first-conjs} and \(\monstr{f^\ast}\) has a conjugate of shape \(y \leftlolli f_*(x) \to f_*(f^\ast(y) \leftlolli x)\), thus of the form \({\leftlolli}\circ(\id \times f_\ast) \nattrans f_\ast \circ {\leftlolli} \circ (f^\ast \times \id)\); this is the following.
\[
    \raisebox{-.45\height}{\input{\figdir/internal_adjunction.pstex_t}}
\]
By Corollary~\ref{cor:two-conjugate-inverse} we know that the natural transformation \(\monstr{f^\ast}\) is invertible if and only its conjugate is, and in that case the inverse of the conjugate is the conjugate of the inverse.  Thus the strong monoidality of \(f^\ast\) corresponds to the above natural transformation giving an internal adjunction between \(f^\ast\) and \(f_\ast\), ie., \(y \leftlolli f_*(x) \cong f_*(f^\ast(y) \leftlolli x)\).  This framework gives a good way of keeping track of the natural transformations that are witnessing such isomorphisms.

\subsection{Strong closed monoidal functors and the projection formula}
As above, if \(f^\ast \colon \cY \to \cX\) is a monoidal functor then there is the structure natural transformation \(\monstr{f^\ast} \colon {\otimes} \circ (f^\ast \times f^\ast) \nattrans f^\ast \circ {\otimes}\).  If \(\cX\) and \(\cY\) are left closed monoidal then we can construct a `mate' 
\(\mate{\monstr{f^\ast}} \colon f^\ast \circ {\leftlolli} \nattrans {\leftlolli} \circ (f^{\ast\op} \times f^\ast)\), so that the components are of the form \(f^\ast(y \leftlolli y') \to f^\ast(y) \leftlolli f^\ast(y')\).  We can call this the closed structure operator.  This is pictured as follows.
\[
    \mate{\monstr{f^\ast}} := \raisebox{-.45\height}{\input{\figdir/closed_monoidal_str_def.pstex_t}}
\]
This natural transformation \(\mate{\monstr{f^\ast}}\) may or may not be invertible, regardless of whether or not \({\monstr{f^\ast}}\) is.  If \(\mate{\monstr{f^\ast}}\) is invertible then we say that \(f^\ast\) is \define{strong left closed monoidal}.  One case of note is when \(\cX\) and \(\cY\) are both \emph{compact} closed monoidal, i.e.~have duals, in that case, \(f^\ast\) being strong monoidal implies that it is strong left closed monoidal.

If, furthermore, \(f^\ast\) has a left adjoint \(f_!\colon \cX \to \cY\) then the monoidal structure transformation \(\monstr{f^\ast} \colon {\otimes} \circ (f^\ast \times \id) \circ (\id \times f^\ast) \nattrans f^\ast \circ {\otimes}\) has a `mate' \(\pi\colon f_! \circ {\otimes} \circ (f^\ast \times \id) \nattrans {\otimes} \circ (\id \times f_!)\) which has components of the form \(f_!(f^\ast(y) \otimes x) \to y \otimes f_!(x)\).  This is known as the \define{left projection operator} or \define{left Hopf operator}~\cite{BruguieresVirelizier:HopfMonadsMonoidal}.  This is pictured, together with its conjugate, as follows.
\[
    \pi \coloneq \raisebox{-.45\height}{\input{\figdir/H_r.pstex_t}}
    \qquad
    \jL(\pi) = \raisebox{-.45\height}{\input{\figdir/pi_conjugate.pstex_t}}
\]
It should be clear that \(\jL(\pi)\) the conjugate of the projection operator is equal to the closed structure operator \(\mate{\monstr{f^\ast}}\).  

When the projection operator is a natural isomorphism, so that  we have
\begin{equation}
    \label{eq:projection-formula}
    \pi\colon f_!(f^\ast(y) \otimes x) \isoarrow y \otimes f_!(x)
\end{equation}
for all $x$ and $y$, it is said that \define{the projection formula holds}.  From Corollary~\ref{cor:two-conjugate-inverse} it is immediate that \(\pi\) is invertible if and only if \(\mate{\monstr{f^\ast}}\) is, with their inverses being conjugate.  This then gives a good conceptual and visual framework for the following known result, where we have given precise meaning to `conjugate'.

\begin{thm}
    Suppose that $f_! \colon \cX \rightleftarrows \cY \cocolon f^\ast$ is an adjunction between left closed monoidal categories such that the right adjoint $f^\ast$ is monoidal, then the projection formula~\eqref{eq:projection-formula} holds if and only if $f^\ast$ is strong left closed monoidal.  In that case the inverse of the left projection operator is given by the conjugate of the closed structure operator.
\end{thm}

\appendix

\section{Proofs of the closed monoidal category statements}
\label{app:proofs}

Here we gives the proofs of the theorems from Section~\ref{subsec:closed-monoidal-cats}.

\begin{proof}[Proof of Theorem~\ref{thm:extranatural-coev-ev}]  To prove that the first diagram commutes, we just fill it with diagrams that are known to commute, as shown in the diagram in Figure~\ref{fig:digram-ev-coev}.
In the diagram one of the triangles commutes trivially, whilst the other triangle commutes due to the triangle identity satisfied by the evaluation and coevaluation owing to them being the counit and unit of an adjunction.  One interior square commutes by the definition of $[f, \secondc]$ given in Section~\ref{subsec:closed-monoidal-cats}; whilst the other two interior squares commute due to the naturality of $\ev_{c,\secondc}$ in $\secondc$.  Thus the exterior square commutes.

\begin{figure}[thb]
    \[
\begin{tikzcd}[ampersand replacement=\&,column sep=large]
{c\otimes [ c', \secondc]} \&\& {c'\otimes [c', \secondc]} \\
\& {c \otimes [c', \secondc]} \\
\& {c\otimes[ c, c\otimes [c',\secondc]]} \\
\& {c\otimes [c, c'\otimes [c', \secondc]]} \\
{c\otimes [c, \secondc]} \&\& \secondc
\arrow["{c \otimes [f,  \secondc]}", tail reversed, no head, from=5-1, to=1-1]
\arrow["{ f\otimes [c',\secondc]}"', tail reversed, no head, from=1-3, to=1-1]
\arrow["{\ev_{c', \secondc}}"', tail reversed, no head, from=5-3, to=1-3]
\arrow["{\ev_{c,\secondc}}", tail reversed, no head, from=5-3, to=5-1]
\arrow["{c \otimes [c, \ev_{c',\secondc}]}"', tail reversed, no head, from=5-1, to=4-2]
\arrow["{c\otimes [c, f \otimes[c',  \secondc]}", tail reversed, no head, from=4-2, to=3-2]
\arrow["{c\otimes \coev_{c,[c', \secondc]}}", tail reversed, no head, from=3-2, to=1-1]
\arrow["\id"', tail reversed, no head, from=2-2, to=1-1]
\arrow["{f \otimes [c', \secondc]}"', tail reversed, no head, from=1-3, to=2-2]
\arrow["{\ev_{c, c'\otimes[c', \secondc]}}"{pos=0.7}, curve={height=-12pt}, tail reversed, no head, from=1-3, to=4-2]
\arrow["{\ev_{c,c \otimes [c', \secondc]}}", tail reversed, no head, from=2-2, to=3-2]
\end{tikzcd}
    \]
    \caption{Diagram for the proof of Theorem~\ref{thm:extranatural-coev-ev}}
    \label{fig:digram-ev-coev}
\end{figure}
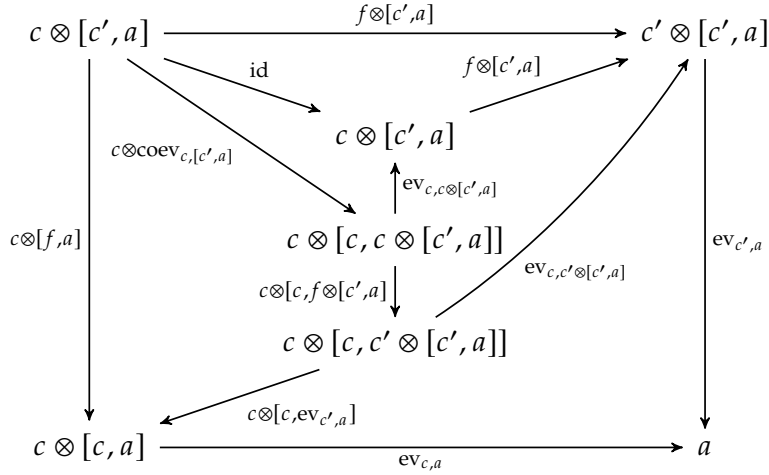

Similarly, the other square can be filled in a dual fashion.
\end{proof}
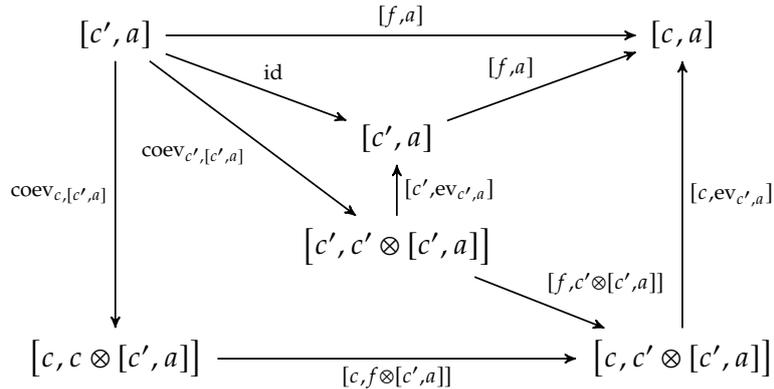
\begin{figure}[tbh]
    \[
\begin{tikzcd}[ampersand replacement=\&]
	{[c',\secondc]} \&\& {[c,\secondc]} \\
	\& {[c',\secondc]} \\
	\& {\bigl[c', c'\otimes [c',\secondc]\bigr]} \\
	{\bigl[c, c\otimes [c',\secondc]\bigr]} \&\& {\bigl[c, c'\otimes [c',\secondc]\bigr]}
	\arrow["{\coev_{c,[c',\secondc]}}"', from=1-1, to=4-1]
	\arrow["{[c, f\otimes [c',\secondc]]}"', from=4-1, to=4-3]
	\arrow["{[f, \secondc]}", from=1-1, to=1-3]
	\arrow["{[c, \ev_{c',\secondc}]}"', from=4-3, to=1-3]
	\arrow["{\coev_{c',[c',\secondc]}}"', from=1-1, to=3-2]
	\arrow["{[c',\ev_{c',\secondc}]}"', from=3-2, to=2-2]
	\arrow["\id", from=1-1, to=2-2]
	\arrow["{[f,\secondc]}", from=2-2, to=1-3]
	\arrow["{[f, c'\otimes [c', \secondc]]}", from=3-2, to=4-3]
\end{tikzcd}
    \]
    \caption{Diagram for the proof of Theorem~\ref{thm:closed-mon-cat-med-brow}}
    \label{fig:diagram-closed-monoidal}
\end{figure}
\begin{proof}[Proof of Theorem~\ref{thm:closed-mon-cat-med-brow}]
    Using \eqref{diag:functor-defn} and Theorem~\ref{thm:extranatural-coev-ev}, it follows that any left closed monoidal category gives rise to the internal hom functor   $[{-}, {-}]$ satisfying the specified conditions.
    
    To show that the data given in the statement of the theorem gives a closed monoidal category it suffices to show that for the data given in the statement, the diagram in~(\ref{diag:functor-defn}) commutes.  This can be show by filling in the diagram with commuting squares and triangles as in Figure~\ref{fig:diagram-closed-monoidal}.
\end{proof}

\section{A double-category and profunctor perspective}
\label{app:double-category}

\enlargethispage*{2em}
In this appendix, more categorical knowledge is assumed. 
We first see a characterization of extranatural transformations in terms of certain natural transformations and then put that in the context of the monoidal double category of functors and profunctors, seeing how the surface diagrams find a natural home there.  We then go on to see how our surface diagrams relate to those of Street in the monoidal bicategory of profunctors, using the fact that we have an equipment.

\subsection{Extranatural transformations as natural transformations}

The definition given of an extranatural transformation, Definition~\ref{defn:extranaturality}, is a useful one and one that absracts key properties of things like evaluation and coevaluation; however, it is not particularly \emph{formal} in that it is not clear how to generalize it to situations other than that of the \(2\)-category of categories.  We can give an alternative, equivalent definition of an extranatural transformation as a natural transformation between profunctors and this is the content of the following theorem.  This leads us below to be able to work more formally with extranatural transformations via the monoidal double category of profunctors and allow us to interpret the surface diagrams there.

This characterization of extranatural transformations is known to experts.  It is implicit in a talk~\cite{Shulman:ExtraordinaryCMSTalk} and a blog post~\cite{Shulman:ExtraordinaryBlogPost} of Shulman, and in a draft introduction~\cite{BaezMellies} to an unwritten paper of Baez and Mellies.  Street~\cite[Section~3]{StreetFunctorialCalculus} indirectly describes the characterization.
\begin{thm}
    \label{thm:extranatural-as-natural}
    Given functors
    \[
        P\colon \cC\times \cC^\op\times \cA\to \cD
        \quad\text{and}\quad
        Q\colon \cA\times \cB^{\op}\times \cB \to \cD
    \]
    there is a canonical one-to-one correspondence between extranatural transformations of the form \(\beta\colon P \extranat Q\) and natural transformations with components of the form
    \[
        \cC(c, c') \times \cA(a, a') \times \cB(b, b')
        \xrightarrow{\hatbeta_{c, c', a, a', b, b'}} 
        \cD\bigl(P(c, c', a), Q(a', b, b')\bigr).
    \]
    This correspondence is constructed as follows.  
    \begin{itemize}
        \item 
        For an extranatural transformation $\beta \colon P \extranat Q$ define the natural transformation component 
        \[
            \hatbeta_{c, c', a, a', b, b'}(f, h, g)\colon P(c, c', a) \to Q(a', b, b')
        \]
        to be any of the composites in the diagram of Definition~\ref{defn:extranaturality} going from the top left to the bottom right.
        \item For a natural transformation $\hatbeta$ of the above form, define the extranatural transformation $\beta \colon P \extranat Q$ by
        \[
            \beta_{c,a,b} := \hatbeta_{c,c,a,a,b,b}(\id_c, \id_a, \id_b) \colon P(c, c, a) \to Q(a, b, b).
        \]
    \end{itemize}
\end{thm}
Before going on to the proof, we should note that the key thing to observe here is that naturality for a natural transformation $\hatbeta$ of the above form means precisely the following. 
If we have strings of morphisms
\[
    c'' \xrightarrow{f_1} c \xrightarrow{f} c' \xrightarrow{f_2} c''',
    \quad
    a'' \xrightarrow{h_1} a \xrightarrow{h} a' \xrightarrow{h_2} a''',
    \quad
    b'' \xrightarrow{g_1} b \xrightarrow{g} b' \xrightarrow{g_2} b''',
\]
in $\cC$, $\cA$ and $\cB$ respectively,
then 
\begin{multline*}
    \hatbeta_{c'',c''',a'',a''',b'',b'''}(f_2 \circ f \circ f_1, h_2 \circ h \circ h_1, g_2 \circ g\circ g_1)\\
    =
    Q(h_2, g_1, g_2) \circ \hatbeta_{c, c', a, a', b, b'}(f, h, g) \circ P(f_1, f_2, h_1).
\end{multline*}
We now prove the theorem.
\begin{proof}
    We need to check that the assignments $\beta \mapsto \hatbeta$ and $\hatbeta \mapsto \beta$ are both well-defined and are mutually inverse.  If we know that they are well-defined then the fact that they are mutually inverse is almost immediate, just requiring a simple use of naturality in one direction.  So it suffices to show that they are well-defined assignments.

    Firstly, we need to show that if $\beta$ is an extranatural transformation then $\hatbeta$ is a natural transformation.  Actually we will just show that it is natural in the second variable as the other cases are similar.  Here the second equality uses the commutativity of the left-hand square in Definition~\ref{defn:extranaturality}.
    \begin{align*}
        \hatbeta_{c, c''', a, a, b, b}&(f_2 \circ f, h, g) 
        \\
        &:= Q(h, \id_b, g)\circ \beta_{c''',a,b} \circ P(f_2\circ f, \id_{c'''}, \id_{a})\\
        &= Q(h, \id_b, g)\circ \beta_{c''',a,b} \circ P(f_2, \id_{c'''}, \id_{a}) \circ P(f, \id_{c'''}, \id_{a})\\
        &= Q(h, \id_b, g)\circ \beta_{c',a,b} \circ P(\id_{c'}, f_2, \id_{a}) \circ P(f, \id_{c'''}, \id_{a})\\
        &= Q(h, \id_b, g)\circ \beta_{c',a,b} \circ P(f, \id_{c'}, \id_{a}) \circ P(\id_{c}, f_2, \id_{a})\\
        &=: \hatbeta_{c, c', a, a, b, b}(f, h, g) \circ P(\id_{c}, f_2, \id_{a}),
    \end{align*}
    as required.

    Finally, we need to show that if $\beta$ is natural then the associated $\beta$ is extranatural.  We will just show that the left-hand square in Definition~\ref{defn:extranaturality} commutes, the commutativity of the other squares are shown similarly.  The two equalities below use naturality of $\beta$ in its first and second arguments.
    \begin{align*}
        \beta_{c', a, b} \circ P(f, \id_{c'}, \id_a) 
        &:= \hatbeta_{c', c', a, a, b, b}(\id_{c'}, \id_{a}, \id_{b}) \circ P(f, \id_{c'}, \id_a)\\
        &= \hatbeta_{c, c', a, a, b, b}(f, \id_{a}, \id_{b}) \\
        &= \hatbeta_{c, c, a, a, b, b}(\id_{c}, \id_{a}, \id_{b}) \circ P(\id_c, f, \id_a)\\
        &=: \beta_{c, a, b} \circ P(\id_c, f, \id_a),
    \end{align*}
    as required.
\end{proof}

\subsection{A double categorical view of surface diagrams}
In Theorem~\ref{thm:extranatural-as-natural} above, we have seen that for functors 
\[
    P\colon \cC\times \cC^\op\times \cA\to \cD
    \quad\text{and}\quad
    Q\colon \cA\times \cB^{\op}\times \cB \to \cD
\]
an extranatural transformation \(\beta\colon P \extranat Q\) is the same as a natural transformation with components of the form
\[
    \cC(c, c') \times \cA(a, a') \times \cB(b, b')
    \to 
    \cD(P(c, c', a), Q(a', b, b')).
\]
 
Now recall that functors and profunctors form a double category.  By double category we mean what is sometimes called a pseudo-double category (see Shulman's paper~\cite{Shulman:FramedBicategories} or Koudenburg's thesis~\cite{Koudenburg:Thesis}).  We will use the convention that a profunctor \(M \colon \cA \profto \cB\) is a functor \(\cA^{\op} \times \cB \to \Set\).  We will write profunctors vertically and functors horizontally, maintaining our convention of right to left and bottom to top.  We will make use of Myers'~\cite{Myers:StringDiagrams} string diagram notation for double categories, so a \(2\)-cell can be notated in the usual arrow notation and in the string diagram notation as follows.
\[
\begin{tikzcd}[baseline={([yshift=-.4em]current bounding box.center)}, ampersand replacement=\&]
	\cD \& \cB \\
	\cC \& \cA
	\arrow[""{name=0, anchor=center, inner sep=0}, "M"', "\shortmid"{marking}, from=2-2, to=1-2]
	\arrow[""{name=1, anchor=center, inner sep=0}, "{M'}", "\shortmid"{marking}, from=2-1, to=1-1]
	\arrow["{G'}"', from=1-2, to=1-1]
	\arrow["G", from=2-2, to=2-1]
	\arrow["\gamma"{description}, draw=none, from=0, to=1]
\end{tikzcd}
\qquad \qquad
\raisebox{-.45\height}{\input{\figdir/sample_double_two_cell.pstex_t}}
\]
In the double category of profunctors, such a \(2\)-cell \(\gamma\) is defined to be a natural transformation with components
\[
    \gamma_{a,b}\colon M(a, b) \to M'\bigl(G(a), G'(b)\bigr).
\]
The \emph{bicategory} of profunctors has duals: the dual of a category \(\cA\) is its opposite category \(\cA^{\op}\), with the evaluation and coevaluation profunctors, \(E_{\cA}\colon \cA \times \cA^{\op} \profto \termcat\) and \(N_{\cA} \colon \termcat \profto  \cA^{\op}  \times \cA\), both represented by the hom functor \(\cA({-}, {-}) \colon  \cA^{\op}  \times \cA \to \Set\), as indeed is the identity profunctor \(\Id_{\cA} \colon \cA \profto \cA\).  Street~\cite{StreetFunctorialCalculus} formalizes such structure as an autonomous bicategory.

Putting all of the above together, we see that an extranatural transformation \(\beta \colon P \extranat Q\), where \(P\) and \(Q\) are as above, is the same as a \(2\)-cell of the following form.
\[
\begin{tikzcd}[baseline={([yshift=-.4em]current bounding box.center)}, ampersand replacement=\&]
	\cD \& {\cA\times \cB^{\op}\times \cB} \\
	\cD \& {\cC\times \cC^{\op} \times \cA}
	\arrow[""{name=0, anchor=center, inner sep=0}, "{E_{\cC} \times \Id_{\cA} \times N_{\cB}}"', "\shortmid"{marking}, from=2-2, to=1-2]
	\arrow[""{name=1, anchor=center, inner sep=0}, "{\Id_{\cD}}", "\shortmid"{marking}, from=2-1, to=1-1]
	\arrow["Q"', from=1-2, to=1-1]
	\arrow["P", from=2-2, to=2-1]
	\arrow["\beta"{description}, draw=none, from=0, to=1]
\end{tikzcd}
\qquad
\raisebox{-.45\height}{\input{\figdir/extranatural_double_string_diag.pstex_t}}
\]
If we take the string diagram and utilize the third dimension going into the page or screen to represent the monoidal structure, as we did in the body of the paper, then we recover the diagram we used for such an extranatural transformations.  We can see how the usual cup and cap notation for the coevaluation and evaluation morphisms in the vertical -- profunctor -- bicategory are used in the right-hand face.
\[
    \raisebox{-.45\height}{\input{\figdir/extranatural_sample.pstex_t}}
\]
From a formal category theory, this means that for a monoidal \(2\)-category \(\mathbb{C}\) the analogue of extranatural transformations can be defined if \(\mathbb{C}\) is the horizontal \(2\)-category of a monoidal double category where the vertical bicategory is monoidal with duals in an appropriate sense, where this needs formalizing.  This is the direction of Street's work~\cite{StreetFunctorialCalculus} though he did not use the double categorical element.  We will now look at some of what he did.

\subsection{The relationship with Street's diagrams}

In this part of the appendix we will see how the diagrams drawn by Street~\cite{StreetFunctorialCalculus} relate to the diagrams used here.  His diagrams look different as, for instance, they involve loops.

Street works in the general context of what he calls an autonomous monoidal bicategory.  We will work concretely in the context of profunctors, but there should be no significant difference to the general setting.

Given profunctors \(M \colon \cC\times \cC^{\op} \times \cA \profto \cD\) and \(M' \colon \cA\times \cB^{\op}\times \cB \profto \cD\), Street defines an \define{extraordinary \(2\)-cell} from \(M\) to \(M'\) to be a \(2\)-cell in the bicategory of profunctors of the form shown on the left below, which we can also depict as the string diagram on the right.
\[
\begin{tikzcd}[ampersand replacement=\&,column sep=tiny]
	\cD \\
	\& {\cA\times \cB^{\op}\times \cB} \\
	{\cC\times \cC^{\op} \times \cA}
	\arrow["{E_{\cC} \times \Id_{\cA} \times N_{\cB}}"'{pos=0.6}, "\shortmid"{marking}, from=3-1, to=2-2]
	\arrow["M"', "\shortmid"{marking}, from=2-2, to=1-1]
	\arrow[""{name=0, anchor=center, inner sep=0}, "{M'}", "\shortmid"{marking}, curve={height=-12pt}, from=3-1, to=1-1]
	\arrow[shorten <=12pt, shorten >=24pt, Rightarrow, from=2-2, to=0]
\end{tikzcd}
    \qquad
    \raisebox{-.45\height}{\input{\figdir/extraordinary_two_cell_string_diagram.pstex_t}}
\] 
Utilizing the third dimension to depict the monoidal product direction, Street represents such an extraordinary \(2\)-cell with the following surface diagram; the top and bottom ridge of the surface are coloured green to emphasize the connection with the above string diagram.
\[
    \raisebox{-.45\height}{\input{\figdir/street_surface_diagram_green.pstex_t}}
\]

To relate this to the notion of extranatural transformation and the surface diagrams we have been using for them, we should recall that the double category of functors and profunctors actually forms an equipment, so each functor \(F\colon \cA \to \cB\) has an associated \define{companion} profunctor \(\companion{F}\colon \cA \profto \cB\), given by \(\companion{F}(a, b) := \cB\bigl(F(a), b\bigr)\), and an associated \define{conjoint} profunctor \(\conjoint{F}\colon \cB \profto \cA\), given by \(\conjoint{F}(b, a) := \cB\bigl(b, F(a)\bigr)\).  We will concentrate on companions at this point.   There are structural \(2\)-cells relating the functor \(F\) to its companion \(\companion{F}\) and, following Myers~\cite{Myers:StringDiagrams}, these are drawn in string diagrams as follows.
\[
    \raisebox{-.45\height}{\input{\figdir/companion_cell_string_1.pstex_t}}
    \qquad\qquad
    \raisebox{-.45\height}{\input{\figdir/companion_cell_string_2b.pstex_t}}
\]
The two ways that these can be composed both give rise to identity \(2\)-cells.

From the characterization of extranatural transformations as double category \(2\)-cells, as in the previous subsection, we can see that an extranatural transformation \(P \extranat Q\) is essentially the same as an extraordinary \(2\)-cell going in the opposite direction between the associated companions, i.e. from \(\companion{Q}\) to \(\companion{P}\).  In string diagrams the correspondence is easy to see.
\[
    \raisebox{-.45\height}{\input{\figdir/extranatural_as_extraordinary_string_1.pstex_t}}
    \quad
    \mapsto
    \quad
    \raisebox{-.45\height}{\input{\figdir/extranatural_as_extraordinary_string_2.pstex_t}}
\]
If we have a functor whose domain is a product, $F\colon \cA \times \cA' \times \cA'' \to \cB$, then the structural \(2\)-cells for its companion profunctor $\companion{F} \colon \cA \times \cA' \times \cA'' \profto \cB$ can be drawn in surface diagrams as follows.
\[
    \raisebox{-.45\height}{\input{\figdir/companion_cell_surface_1.pstex_t}}
    \qquad\qquad
    \raisebox{-.45\height}{\input{\figdir/companion_cell_surface_2.pstex_t}}
\]
So now when the string diagram from the right hand side above is expanded out of the plane to form a surface diagram we get the following.
\[
    \raisebox{-.45\height}{\input{\figdir/street_surface_plus_companions_new.pstex_t}}
\]
Thus we recover our standard way of representing an extranatural transformation and see how it relates to the style of diagram used by Street.



\enlargethispage*{2em}

\printbibliography

@phdthesis{Aravanis:thesis,
	author = {Aravanis, Christos},
	school = {University of Sheffield},
	title = {{H}opf algebras, {H}opf monads and derived categories of sheaves},
	year = {2018},
	url = {https://etheses.whiterose.ac.uk/23287/}
}

@unpublished{AravanisWillerton:Hopf-monads,
	author = {Aravanis, Christos and Willerton, Simon},
	title = {{H}opf algebras, {H}opf monads and derived categories of sheaves},
	note = {in preparation},
}

@Misc{Arkor:quiver,
  author = {Nathanael Arkor},
  note   = {graphical editor for commutative diagrams},
  title  = {{quiver}},
  url    = {https://q.uiver.app/},
}

@Article{BaezDolan:HDA-III,
  author     = {Baez, John C. and Dolan, James},
  journal    = {Advances in Mathematics},
  title      = {Higher-Dimensional Algebra {III}. {$n$}-Categories and the Algebra of Opetopes},
  year       = {1998},
  issn       = {0001-8708},
  month      = may,
  number     = {2},
  pages      = {145--206},
  volume     = {135},
  _file      = {:Baez_1998 - Higher Dimensional Algebra III.n Categories and the Algebra of Opetopes.pdf:PDF},
  _publisher = {Elsevier BV},
  doi        = {10.1006/aima.1997.1695},
}

@unpublished{BaezMellies,
  author = {Baez, John C. and Melli\`es, Paul-Andr\'e},
  month  = mar,
  _note   = {`draft introduction'},
  title  = {Theories With Duality (DRAFT VERSION ONLY)},
  year   = {2010},
  _file   = {:Baez2010 - Theories with Duality (DRAFT VERSION ONLY).pdf:PDF},
  url    = {https://math.ucr.edu/home/baez/duality.pdf},
}

@article{BruguieresVirelizier:HopfMonadsMonoidal,
	Author = {Brugui{\`e}res, Alain and Lack, Steve and Virelizier, Alexis},
	Title = {Hopf monads on monoidal categories},
	_Date-Added = {2017-09-06 10:11:38 +0000},
	_Date-Modified = {2017-09-06 10:13:39 +0000},
	Doi = {10.1016/j.aim.2011.02.008},
	Journal = {Advances in Mathematics},
	_Issn = {0001-8708},
	_Journal = {Adv. Math.},
	_Mrclass = {18C20 (16T05 18D10 18D15)},
	_Mrnumber = {2793022},
	_Mrreviewer = {Joost Vercruysse},
	Number = {2},
	Pages = {745--800},
	_Url = {http://dx.doi.org/10.1016/j.aim.2011.02.008},
	Volume = {227},
	Year = {2011},
	_Bdsk-Url-1 = {http://www.ams.org/mathscinet-getitem?mr=2793022},
	}

@Article{ChengGurskiRiehl,
  author    = {Eugenia Cheng and Nick Gurski and Emily Riehl},
  journal   = {Journal of K-Theory},
  title     = {Cyclic multicategories, multivariable adjunctions and mates},
  year      = {2014},
  _month     = {jan},
  number    = {2},
  pages     = {337--396},
  volume    = {13},
  _abstract  = {A multivariable adjunction is the generalisation of the notion of a 2-variable adjunction, the classical example being the hom/tensor/cotensor trio of functors, to n C 1 functors of n variables. In the presence of multivariable adjunctions, natural transformations between certain composites built from multivariable functors have "dual" forms. We refer to corresponding natural transformations as multivariable or parametrised mates, generalising the mates correspondence for ordinary adjunctions, which enables one to pass between natural transformations involving left adjoints to those involving right adjoints. A central problem is how to express the naturality (or functoriality) of the parametrised mates, giving a precise characterization of the dualities so-encoded. We present the notion of "cyclic double multicategory" as a structure in which to organise multivariable adjunctions and mates. While the standard mates correspondence is described using an isomorphism of double categories, the multivariable version requires the framework of "double multicategories". Moreover, we show that the analogous isomorphisms of double multicategories give a cyclic action on the multimaps, yielding the notion of "cyclic double multicategory". The work is motivated by and applied to Riehl's approach to algebraic monoidal model categories.},
  doi       = {10.1017/is013012007jkt250},
  _eprint    = {1208.4520},
  _eprinttype = {arxiv},
  _file      = {:/Users/simonwillerton/Downloads/mates.pdf:PDF},
  _keywords  = {Adjunctions, mates, multicategories Mathematics Subject Classification 2010: 18A40, 18D99},
  publisher = {Cambridge University Press ({CUP})},
}

@Article{EilenbergKelly,
  author       = {Samuel Eilenberg and G.M Kelly},
  journal      = {Journal of Algebra},
  title        = {A generalization of the functorial calculus},
  year         = {1966},
  _month        = {5},
  number       = {3},
  pages        = {366--375},
  volume       = {3},
  creationdate = {2023-07-21T09:54:28},
  doi          = {10.1016/0021-8693(66)90006-8},
  file         = {:Eilenberg_1966 - A Generalization of the Functorial Calculus.pdf:PDF},
  publisher    = {Elsevier {BV}},
}

@article{FauskHuMay:Isomorphisms,
	Author = {Fausk, H. and Hu, P. and May, J. P.},
	Pages = {107--131},
	Title = {Isomorphisms between left and right adjoints},
	Journal = {{Theory and Applications of Categories}},
	Volume = {11},
	Number = {4},
	Year = {2003},
	url = {http://www.tac.mta.ca/tac/volumes/11/4/11-04abs.html},
	_Date-Modified = {2018-05-11 13:28:35 +0000},
	_Fjournal = {Theory and Applications of Categories},
	_Issn = {1201-561X},
	_Journal = {Theory Appl. Categ.},
	_Mrclass = {18D10 (14A10 55U35)},
	_Mrnumber = {1988072},
	_Mrreviewer = {Gabriele Vezzosi},
	}

@Book{HoveyModelCategories,
  author    = {Hovey, Mark},
  publisher = {American Mathematical Society},
  title     = {Model Categories},
  year      = {2007},
  _month     = {oct},
  doi       = {10.1090/surv/063},
  _file      = {:Hovey_2007 - Model Categories.pdf:PDF},
}

@Article{KanAdjointFunctors,
  author    = {Daniel M. Kan},
  journal   = {Transactions of the American Mathematical Society},
  title     = {Adjoint functors},
  year      = {1958},
  number    = {2},
  pages     = {294--329},
  volume    = {87},
  doi       = {10.1090/s0002-9947-1958-0131451-0},
  _file      = {:Kan_1958 - Adjoint Functors.pdf:PDF},
  _publisher = {American Mathematical Society ({AMS})},
}

@InCollection{KellyStreetReview2Cats,
  author       = {G. M. Kelly and Ross Street},
  booktitle    = {Category Seminar},
  publisher    = {Springer Berlin Heidelberg},
  title        = {Review of the elements of 2-categories},
  year         = {1974},
  pages        = {75--103},
  _creationdate = {2023-08-18T12:15:06},
  doi          = {10.1007/bfb0063101},
  _file         = {:Kelly1974 - Review of the Elements of 2 Categories.pdf:PDF},
}

@PhdThesis{Koudenburg:Thesis,
  author    = {Koudenburg, Seerp Roald},
  title     = {Algebraic weighted colimits},
  year      = {2013},
  _doi       = {10.48550/ARXIV.1304.4079},
  keywords  = {Category Theory (math.CT), FOS: Mathematics, FOS: Mathematics},
  eprint = {1304.4079},
  eprinttype = {arxiv},
}

@Book{MacLaneCategoriesWork,
  author    = {Mac Lane, Saunders},
  publisher = {Springer New York},
  title     = {Categories for the Working Mathematician},
  year      = {1998},
  edition   = {2},
  _isbn      = {978-0-387-98403-2},
  series    = {Graduate Texts in Mathematics},
  volume    = {5},
  doi       = {10.1007/978-1-4757-4721-8},
  _file      = {:Mac_Lane_1978 - Categories for the Working Mathematician.pdf:PDF},
}

@Article{Myers:StringDiagrams,
  author    = {Myers, David Jaz},
  title     = {String Diagrams For Double Categories and Equipments},
  year      = {2016},
  _copyright = {arXiv.org perpetual, non-exclusive license},
  _doi       = {10.48550/arXiv.1612.02762},
  _file      = {:Myers2016 - String Diagrams for Double Categories and Equipments.pdf:PDF},
  _groups    = {reading_list},
  _keywords  = {Category Theory (math.CT), FOS: Mathematics, FOS: Mathematics, 18A99},
  _publisher = {arXiv},
  eprint = {1612.02762},
  eprinttype = {arxiv},
}

@misc{nlab:closed_monoidal_category,
  author = {{nLab authors}},
  title = {closed monoidal category},
  howpublished = {\url{https://ncatlab.org/nlab/show/closed+monoidal+category}},
  note = {\href{https://ncatlab.org/nlab/revision/closed+monoidal+category/52}{Revision 52}},
  _month = aug,
  year = 2023
}

@Article{Shulman:FramedBicategories,
  author        = {Shulman, Michael A.},
  journal       = {Theory and Applications of Categories},
  title         = {Framed bicategories and monoidal fibrations},
  year          = {2008},
  month         = jun,
  number        = {18},
  pages         = {650--738},
  volume        = {20},
  _archiveprefix = {arXiv},
  _copyright     = {arXiv.org perpetual, non-exclusive license},
  doi           = {10.48550/ARXIV.0706.1286},
  _eprint        = {https://arxiv.org/abs/0706.1286},
  _file          = {:Shulman2008 - Framed Bicategories and Monoidal Fibrations.pdf:PDF},
  _keywords      = {Category Theory (math.CT), FOS: Mathematics, 18D05 (Primary), 18D30, 18D10 (Secondary)},
  _primaryclass  = {math.CT},
  _publisher     = {arXiv},
  url           = {http://www.tac.mta.ca/tac/volumes/20/18/20-18abs.html},
}

@Misc{Shulman:ExtraordinaryBlogPost,
  author       = {Mike Shulman},
  howpublished = {$n$-Category Caf\'e blog post},
  month        = {3},
  title        = {Extraordinary $2$-Multicategories},
  year         = {2010},
  url          = {https://golem.ph.utexas.edu/category/2010/03/extraordinary_2multicategories.html},
}

@Misc{Shulman:ExtraordinaryCMSTalk,
  author = {Mike Shulman},
  note   = {talk at 2010 Canadian Mathematical Society Summer Meeting},
  title  = {Extraordinary multicategories},
  year   = {2010},
  url    = {https://home.sandiego.edu/~shulman/papers/cms2010talk.pdf},
}

@Article{StreetFunctorialCalculus,
  author    = {Ross Street},
  journal   = {Applied Categorical Structures},
  title     = {Functorial Calculus in Monoidal Bicategories},
  year      = {2003},
  number    = {3},
  pages     = {219--227},
  volume    = {11},
  _abstract  = {The definition and calculus of extraordinary natural transformations is extended to a context internal to any autonomous monoidal bicategory. The original calculus is recaptured from the geometry of the monoidal bicategory V-Mod whose objects are categories enriched in a cocomplete symmetric monoidal category V and whose morphisms are modules.},
  doi       = {10.1023/a:1024247613677},
  _file      = {:Calculus2003 - Functorial Calculus in Monoidal Bicategories.pdf:PDF},
  _keywords  = {Mathematics Subject Classifications (2000): 18D05, 18D20, 18D10 enriched category, dual, dinatural transformation, Gray monoid},
  publisher = {Springer Science and Business Media {LLC}},
}

@article{Willerton:HopfMonads,
	Author = {Willerton, Simon},
	Title = {A diagrammatic approach to {H}opf monads},
	Journal = {The Arabian Journal for Science and Engineering. Section C.},
	eprint = {0807.0658},
	eprinttype = {arxiv},
	Number = {2},
	Pages = {561--585},
	Volume = {33},
	Year = {2008},
}

@Misc{xfig,
  note  = {diagramming tool},
  title = {Xfig},
  url   = {https://sourceforge.net/projects/mcj/},
}

\end{document}